\documentclass{article}

\usepackage{arxiv}

\usepackage[utf8]{inputenc} % allow utf-8 input
\usepackage[T1]{fontenc}    % use 8-bit T1 fonts
\usepackage{hyperref}       % hyperlinks
\usepackage{url}            % simple URL typesetting
\usepackage{booktabs}       % professional-quality tables
\usepackage{amsfonts}       % blackboard math symbols
\usepackage{nicefrac}       % compact symbols for 1/2, etc.
\usepackage{microtype}      % microtypography
\usepackage{lipsum}

\usepackage{amsmath,amsfonts,amssymb,amsthm}

\usepackage{graphicx}
\usepackage{subfigure}
\usepackage[super]{nth}
\usepackage{tabularx}
\usepackage{graphicx}
\usepackage{adjustbox}
\usepackage{multirow}
\usepackage{algorithm}
\usepackage{algorithmicx}
\usepackage{algcompatible}
\usepackage{multicol}
\usepackage{algpseudocode}
\usepackage[super]{nth}
\usepackage{enumitem}
\usepackage{xcolor}
\usepackage{outline}
\newtheorem{defn}{Definition}
\newtheorem{thm}{Theorem}
\newtheorem{lem}{Lemma}
\newtheorem{prop}{Proposition}

\newtheorem{assm}{Assumption}

\newcommand{\E}{\mathbb{E}}
\newcommand{\Q}{\mathfrak{Q}}

\newcommand{\exclude}[1]{}

\newcommand{\X}{\mathcal{X}}

\newcommand{\Tau}{\mathrm{T}}
\renewcommand{\vec}[1]{\mathbf{#1}}

\def\argmin{\mathop{\rm arg\,min}}%

\title{Rolling Horizon Policies in Multistage Stochastic Programming}

\author{
Murwan Siddig \\
  Department of Industrial Engineering\\
  Clemson University\\
  Clemson, SC, USA 29631  \\
  \texttt{msiddig@clemson.edu} \\
   \And
 Yongjia Song \\
  Department of Industrial Engineering\\
  Clemson University\\
  Clemson, SC, USA 29631  \\
  \texttt{yongjis@clemson.edu} \\
  \And
   Amin Khademi \\
  Department of Industrial Engineering\\
  Clemson University\\
  Clemson, SC, USA 29631  \\
  \texttt{khademi@clemson.edu} \\
}

\begin{document}
\maketitle
\begin{abstract}
Multistage Stochastic Programming (MSP) is a class of models for sequential decision-making under uncertainty. MSP problems are known for their computational intractability due to the sequential nature of the decision-making structure and the  uncertainty in the problem data due to the so-called curse of dimensionality. A common approach to tackle MSP problems with a large number of stages is a rolling-horizon (RH) procedure, where one solves a sequence of MSP problems with a smaller number of stages. This leads to a delicate issue of how many stages to include in the smaller problems used in the RH procedure. This paper addresses this question for, both, finite and infinite horizon MSP problems. For the infinite horizon case with discounted costs, we derive a bound which can be used to prescribe an $\epsilon-$sufficient number of stages. For the finite horizon case, we propose a heuristic approach from the perspective of approximate dynamic programming to provide a sufficient number of stages for each roll in the RH procedure. Our numerical experiments on a hydrothermal power generation planning problem show the effectiveness of the proposed approaches. 
\end{abstract}

% keywords can be removed

\section{Introduction}
\label{sec:intro}
Multistage Stochastic Programming (MSP) is a class of decision-making models where the decision-maker (DM) may adapt and control the behavior of a \textit{probabilistic} system \textit{sequentially} over multiple stages. The goal of the DM is to choose a decision policy that leads to optimal system performance with respect to some performance criterion, e.g., maximizing the expected profit or minimizing the expected cost. MSP models have many real-life applications; these include (but not limited to) energy~\cite{goel2004stochastic,de2017assessing,homem2011sampling,pereira2005strategic,rebennack2016combining,shapiro2013risk}, finance~\cite{ahmed2003multi,ch2007modeling,dupavcova2009portfolio,dupavcova2009asset}, transportation~\cite{alonso2000stochastic,fhoula2013stochastic,herer2006multilocation} and sports~\cite{pantuso2017football}, among others. 

Two main challenges arise from solving MSP problems: (i) the \textit{uncertainty} in the problem data; and (ii) the \textit{nested} form in the sequential decision-making structure of the problem. A typical approach to tackle these challenges is to \textit{approximate} the underlying stochastic process using scenario trees to address the former and make use of the dynamic programming principle via the so-called \emph{Bellman equations}~\cite{bellman} to address the latter. This approach, however, could become quickly impractical -- especially, with the increase in the dimensionality of the state variables ($s_t$) of the system, and the number of stages ($T$) in the planning horizon. This is known as the curse-of-dimensionality and it gives rise to an exponential growth in computational resources required to solve such large MSP problems under certain optimality guarantees. In this paper, we are concerned with the computational tractability of MSP problems with a large number of stages $T$, including the special case when $T = \infty$.

In practice, a typical workaround for the computational intractability of MSP problems with a large number of stages is to solve a sequence of MSP problems in an ``online'' fashion, each of which has a smaller number of stages. This approach is the so-called \emph{rolling horizon} (RH) procedure, and the (smaller) number of stages used in each roll of the RH procedure is referred to as the \textit{forecast horizon}~\cite{chand2002forecast,hernandez1988forecast}. In this procedure, the DM fixes a forecast horizon with $\tau \leq T$, solves the corresponding $\tau$-stage problem, implements the decisions \emph{only} for the current period, rolls forward one period, and repeats the process starting from a new initial state. Nevertheless, because the length of the forecast horizon $\tau$ employed by the RH procedure is typically smaller than the actual planning horizon of the original MSP problem, the resulting decision policy, which is referred to as the \emph{look-ahead policy}, may not be optimal. It is conceivable that the degree of suboptimality in a look-ahead policy is tied to the choice for the number of stages $\tau$ used in the RH procedure. This leads to a very delicate issue of how to choose a small enough forecast horizon with a tolerable optimality gap, which we refer to as a \emph{sufficient} forecast horizon $\tau^*$ (see Definition~\ref{defn:sufficient_look-ahead}). The goal of this paper is to address the issue of finding $\tau^*$ in a RH procedure.

In the literature, there is a good deal of work on the RH procedure, aiming to achieve a balance between computational efficiency and the corresponding policy performance. In~\cite{maggioni2014bounds}, for instance, instead of solving a sequence of MSPs with small forecast horizons, the authors use a deterministic approximation to the MSP in the RH procedure using the point forecasts of all the exogenous uncertain future information. The authors also provide bounds on the suboptimality of the corresponding solution and \cite{maggioni2016bounds} extend these bounds to a stochastic RH approximation using what is known as the reference scenarios. ~\cite{pantuso2019number} use a partial approximation approach which uses the point forecasts of exogenous future information only beyond a certain number of stages until the end of the horizon to truncate the number of stages considered in the MSPs. Another common approach is the so-called stage aggregation, where the DM aggregates all the decisions to be made in multiple stages into a single decision stage (see the discussion in~\cite{powell2014clearing}). This approach is used by~\cite{zou2018partially}, who introduced the \emph{partially} adaptive MSP model where stage aggregations only occur after a certain period until the end of the horizon. 

Unlike these prior works, we do not consider aggregating any of the future exogenous uncertain information nor the stages. Instead, we consider solving MSP problems with a smaller number of stages and provide a systematic way of detecting a sufficient forecast horizon $\tau^*$. We focus on finite/infinite horizon MSP problems where the stochastic process characterizing the data uncertainty is stationary. In the \emph{infinite} horizon \emph{discounted} case, given a fixed forecast horizon $\tau$, we show that the resulting optimality gap associated with the (\textit{static}) look-ahead policy in terms of the total expected discounted reward/cost can be bounded by a function of the forecast horizon $\tau$. This function can be used to derive a forecast horizon $\tau^*_{\epsilon}$ which achieves a prescribed $\epsilon$ optimality gap. In the \emph{finite} horizon case (with no discount), we take an approximate dynamic programming (ADP)~\cite{powell2011approximate} perspective and develop a \textit{dynamic} look-ahead policy where the forecast horizon $\tau$ to use in each roll is chosen dynamically according to the state of the system at that stage. 

The rest of this paper is organized as follows. In Section~\ref{sec:prelim}, we describe in detail the key components of an MSP model, introduce the necessary mathematical notations, and discuss some basic assumptions and solution approaches. In Section~\ref{sec:error_bound} we provide a detailed description of the RH procedure and the main theoretical result for the infinite horizon discounted case. In Section~\ref{sec:ADP_approach} we present the ADP-based heuristic approach for constructing the dynamic look-ahead policy for the finite horizon case. An extensive numerical experiment analysis to the proposed approaches on a multi-period hydro-thermal power generation planning problem is presented in Section~\ref{sec:numerical_results}. Finally, in Section~\ref{sec:conclusion}, we conclude with some final remarks.

\section{Preliminaries on Multistage Stochastic Programming}
\label{sec:prelim}
This section is organized as follows. First, we introduce a generic nested formulation for a finite horizon MSP problem, its dynamic programming (DP) counterpart, and some basic assumptions. Second, we discuss solution methodology based on the DP formulation under these basic assumptions. Finally, we describe a stationary analog for the discounted infinite horizon MSP problems. 
\subsection{Generic MSP formulations and some basic assumptions}
\label{subsec:generic_MSP}
Consider a generic nested formulation for a finite horizon MSP problem as follows:
\begin{footnotesize}
\begin{equation}
\label{eq:MSP}
\min_{x_1 \in \X_1(x_0,\xi_1)} f_1(x_1, \xi_{1})+ \E_{|\xi_{[1]}}\left[\rule{0cm}{0.4cm} \min_{x_2\in \X_2(x_1,\xi_2)} f_2(x_2,\xi_{2}) + \E_{|\xi_{[2]}} \left[\rule{0cm}{0.4cm}\cdots + \E_{|\xi_{[T-1]}}\left[\rule{0cm}{0.4cm} \min_{x_T\in \X_T(x_{T-1},\xi_T)}f_T(x_T,\xi_{T}) \right]\right]\right].
\end{equation}
\end{footnotesize}

Here, the initial vector $x_0$ and initial data $\xi_1$ are assumed to be deterministic, and the exogenous random data is modeled by a stochastic process $(\xi_2, \dots, \xi_T) \in \Xi_2 \times \cdots \times \Xi_T$, where each random vector $\xi_t$ is associated with a known probability distribution $D_{t}$ supported on a set $\Xi_{t} \subset \mathbb{R}^{n}$. Moreover, we denote the history of the stochastic process up to stage $t$ by $\xi_{[t]} := (\xi_1, \dots, \xi_t)$. The decision variable $x_t$ is also referred to as the pre-decision state, and the state of the system $s_t$ in stage $t$ is fully resolved after the realization of the random vector $\xi_t$ is observed. Here, $s_t:= S_t(x_t,\xi_t)$ with $S_t:\X_t\times\Xi_t \to \mathbb{R}^m$.

In its present form, this formulation has two main challenges: (i) the nested optimization posed by the sequential nature in the decision-making structure; and (ii) the (conditional) expectation posed by the stochastic nature of the problem. The first challenge can be addressed by using a DP perspective via the Bellman equations and the computation simplifies dramatically under the following \textit{stage-wise} independence assumption.
\begin{assm}
\label{assm:stagewise-indp}
(Stage-wise independence) We assume that the stochastic process $\{\xi_t\}$ is stage-wise independent, i.e., $\xi_t$ is independent of the history of the stochastic process up to time $t-1$, for $t=1, 2, \dots T$, which is given by $\xi_{[t-1]}$.
\end{assm}

Under this assumption, problem~\eqref{eq:MSP} can be written as:
\begin{equation}
\label{eq:1stagepbm}
\displaystyle \min_{x_1\in \X_1(x_0,\xi_1)} f_1(x_1,\xi_{1}) + \Q_2(x_1),
\end{equation}
where for $t=1, \dots, T-1$: $\Q_{t+1}(x_t):= \E[Q_{t+1}(x_t,\xi_{t+1})]$ is referred to as the expected \emph{cost-to-go} function, 
\begin{equation}
\label{eq:cost2go}
Q_t(x_{t-1},\xi_{t}):=
\begin{array}{llll}
\displaystyle\min_{x_t\in \X_t(x_{t-1},\xi_t)} & f_t(x_t,\xi_{t}) + \Q_{t+1}(x_t),
\end{array}
\end{equation}
and $\Q_{T+1}(x_{T}):=0$. As for the second challenge, a typical approach is to proceed by means of discretization and/or approximate the (discretized) distribution of $\xi_t$ by a sample $N_t$ obtained using sampling techniques such as Monte Carlo methods. This is, for instance, the case in which~\eqref{eq:1stagepbm} is a \emph{sample average approximation} (SAA) where the expectation $\E[Q_{t+1}(x_t,\xi_{t+1})]$ is approximated by a sample average. We refer the reader to \cite{shapiro2011analysis} for a detailed discussion on this topic. Before we discuss the solution methodology based on the DP formulation~\eqref{eq:1stagepbm}, it is important to make the following additional assumptions.
\begin{assm}
\label{assm:relative_recourse}
(Relatively complete recourse) We assume that $\X_{t}(x_{t-1},\xi_t) \neq \emptyset, \ \forall x_{t-1} \in \X_{t-1}$ and $\xi_{t} \in \Xi_{t}, \ \forall t=1, \dots, T$. 
\end{assm}
\begin{assm}
\label{assm:boundedness}
(Boundedness) We assume that the immediate cost function $f_t(\cdot,\cdot)$ at each stage $t$ is bounded, i.e., $\exists \, \kappa >0$ such that $|f_t(x_t,\xi_t)|\leq \kappa, \, \forall (x_t,\xi_t) \in \mathcal{X}_t\times\Xi_{t}, \; \forall \, t=1, 2 \dots, T$.
\end{assm}
\begin{assm}
\label{assm:linearity}
We consider multistage stochastic \emph{linear} programs (MSLPs), i.e., $f_t(x_t,\xi_{t})$ in~\eqref{eq:MSP} is defined as: 
\begin{equation}
\label{eq:MLSPobj}
f_t(x_t,\xi_{t}) :=
\begin{cases}
\vec{\tilde c}_{t}^\top x_t  \quad \text{if} \; x_t \geq 0, \quad \forall t=1, \dots, T\\
+\infty \quad \text{otherwise.}
\end{cases}
\end{equation}
and $\X_t(x_{t-1}, \xi_t) := \{x_t \in \mathbb{R}_{+}^m\; | \; \tilde A_t x_t + \tilde B_t x_{t-1} = \tilde b_t\}$ with $\xi_t = (\vec{\tilde c}_{t}, \tilde A_t, \tilde B_t, \tilde b_t)$ for $t=1, \dots T$.
\end{assm}

Assumption~\ref{assm:relative_recourse} and Assumption~\ref{assm:linearity} are made for simplicity -- the proposed approaches can be potentially used for more general classes of MSPs. Assumption~\ref{assm:boundedness} is a technical assumption made without loss of generality. \exclude{The numerical tractability of \emph{linear programs} (LPs) makes the SAA approximation by scenarios a very attractive approach. This is because the SAA approximation has a deterministic equivalent program (DEP) with a special structure that lends itself to decomposition schemes developed for solving large-scale LPs. These include variants of \emph{Benders} decomposition (also called the L-shaped method \cite{Lshaped}) for the two-stage setting, which generalizes to \emph{nested Benders} \cite{rahmaniani2017benders} for the multistage case. These decomposition schemes can also be viewed as variants of Kelley's cutting plane method for solving convex programs \cite{kelley1960}. }

\subsection{The stochastic dual dynamic programming (SDDP) algorithm}
\label{subsec:SDDP}
Under Assumption~\ref{assm:linearity}, a backward induction argument can be used to show that the expected cost-to-go function $\Q_{t}(x_{t-1})$  is piecewise linear and convex with respect to $x_{t-1}, \; \forall t=1, \dots, T$. Therefore, $\Q_{t}(x_{t-1})$ can be approximated from below by an \textit{outer} cutting-plane approximation $\check \Q_{t}(x_{t-1})$, which can be represented by the maximum of a collection of $L$ cutting planes (cuts):
\begin{equation}
\label{eq:cost_to_go_approx}
\check\Q_{t}(x_{t-1}) = \max_{\ell \in L}\left\{\beta_{t, \ell}^{\top} x_{t-1} + \alpha_{t, \ell} \right\}.
\end{equation}

A common approach for assembling these \textit{cuts} is the SDDP algorithm~\cite{SDDP}. Drawing influence from the backward recursion technique developed in DP, the SDDP algorithm alternates between two main steps: (i) a forward simulation (\emph{forward pass}) which uses the current approximate expected cost-to-go functions $\check \Q_{t+1}(\cdot)$ to generate a sequence of decisions $\check x_t := \check x_t(\xi_t), \forall \; t=2, \dots, T$; and (ii) a backward recursion (\emph{backward pass}) to improve the approximation $\check \Q_{t+1}(\cdot)$, for $t=2, \dots T$. After the forward step, a statistical upper bound for the optimal value of~\eqref{eq:1stagepbm} can be computed; and after the backward step, an improved lower bound for the optimal value of~\eqref{eq:1stagepbm} is obtained. \exclude{Nevertheless, an important assumption in the discussion of the SDDP algorithm is the following \emph{stage-wise} Independence assumption.{\color{blue} (Please check, but this assumption needs to be moved to somewhere above. It is the stage-wise independence assumption that makes it possible to define an expected cost-to-go function $\Q_t(\cdot)$ one for each stage rather than one for each node, e.g., in a scenario tree. Some reorganization needs to be done here after you make the move.)}} We summarize the forward and backward steps of the algorithm next and the reader is referred to \cite{SDDP} for a detailed discussion on that topic.

Given the initial (pre-decision) state $x_0$, the initial realization $\xi_1$, and the current approximations of the expected cost-to-go functions $(\check \Q_2(\cdot), \dots, \check \Q_T(\cdot))$, the SDDP algorithm proceeds as follows.
\begin{itemize}
  \item \textit{Forward pass.}
  \begin{enumerate}
  \item Initialize a list of candidate solutions $(\check x_1, \dots, \check x_T)$, set $t=1$, $\check x_{t-1} = x_0$, $\check \xi_t = \xi_1$.
  \item Solve $\check Q_t(\check x_{t-1},\check\xi_t)$ with $\check \Q_{t+1}(\cdot)$ as shown in~\eqref{eq:cost2go}, to obtain $x^*_t$ and then set $\check x_t = x^*_t$.
  \item If $t=T$ go to \textit{Backward pass}. Otherwise, let $t \gets t+1$, sample a new realization $\xi_{t+1} \in \Xi_{t+1}$ with probability $\mathbb{P}(\xi_{t+1} \in \Xi_{t+1})$, set $\check \xi_{t+1} = \xi_{t+1}$, and go to step 2.
  \end{enumerate}
  \item \textit{Backward pass.}
  \begin{enumerate}
  \item Given $(\check x_1, \dots, \check x_T)$, for $t=T,\dots,2$ do the following.
  \item For every $\xi_t \in \Xi_t$, solve $\check Q_t(\check x_{t-1},\xi_t)$ with $\check \Q_{t+1}(\cdot)$ as shown in~\eqref{eq:cost2go}. Add a cut weighted by the respective probability $\mathbb{P}(\xi_{t} \in \Xi_{t})$ to $\check \Q_{t}(\cdot)$ (if violated by $\check{x}$).
  \item If the termination criterion is met, \textbf{STOP} and return the expected cost-to-go functions $\check \Q_{t}(\cdot),\; \forall t =2,\dots, T$. Otherwise, go to \textit{Forward pass} and repeat.
  \end{enumerate}
\end{itemize}

While the SDDP algorithm has become a popular approach for solving MSP problems, using SDDP (and other similar decomposition schemes) can become impractical with the increase in the dimensionality of the problem. In the literature, there has been some recent advancement on how some of this impracticality can be mitigated by (static/dynamic) scenario reduction and aggregation techniques, see, e.g.,~\cite{siddig2019adaptive} and the references therein. Nevertheless, even when using such enhancement techniques, the viability of these approaches could be easily undermined in MSP problems where the number of stages in the planning horizon $T$ is large. 

\subsection{Stationary discounted infinite horizon MSP problems}
\label{subsec:infinite_MSP}
Consider an infinite horizon MSP problem of the form~\eqref{eq:MSP} introduced in~\cite{shapiro2019stationary}, where $T=\infty$ and the stochastic process has a periodic behavior with $m \in \mathbb{N}$ periods. To simplify the presentation, we will only consider the \textit{risk-neutral} and \textit{stationary} case where the number of periods $m=1$. To that end, we have the following \textit{stationarity} assumption. 
\begin{assm}
\label{assm:stationarity}
(Stationarity) The random vector $\xi_{t}$ has the same distribution with support $\Xi \subset \mathbb{R}^m$, and functions $f_t(\cdot,\cdot), A_t(\cdot), B_t(\cdot), b_t(\cdot)$ are identical for all $t = 2,3,\ldots, T$.
\end{assm}

Unlike the finite horizon case, where the solvability of the DP formulation~\eqref{eq:cost2go} is guaranteed by the finiteness of the terminal stage value function, infinite horizon MSP problems require an evaluation of an infinite sequence of costs for every state $s_t$. As such, it is necessary to establish some notion of finiteness for the expected cost-to-go functions $\Q_t(\cdot)$. One common approach for doing this is to introduce a \emph{discount} factor $\gamma \in (0,1)$ and find a policy $\pi =\{x^{\pi}_t\}_{t=1}^{T}$ which minimizes the expected total discounted cost, that is,
\begin{equation}
\label{eq:expec_tot_disc_cost}
\min_{\pi \in \Pi} f_1(x^{\pi}_1, \xi_1) + \lim_{T\to \infty} \mathbb{E}\left[\sum_{t=2}^{T} \gamma^{t-1} f_t(x^{\pi}_t, \xi_t)\right].
\end{equation}
Note that, since $\xi_t$'s are assumed to follow the same probability distribution for all $t\geq 2$, we can omit the stage subscript $t$ and obtain the following discounted infinite horizon stationary variant of the DP formulation~\eqref{eq:cost2go}:
\begin{equation}
  \label{eq:cost2go_periodic}
  Q(x,\xi):=
  \begin{array}{llll}
  \displaystyle\min_{x'\in \X(x,\xi)} & f(x', \xi) + \gamma\Q(x'),\\
  \end{array}
\end{equation}
where $\Q(x):= \E[Q(x,\xi)]$ and random vector $\xi$ has the same distribution as $\xi_t$'s.

The popularity of these discounted models for infinite horizon MSPs stems from their numerical solvability and their wide range of direct applications in real-life economic problems where the DM accounts for the time value of future costs. This solvability, however, relies on the existence of a \emph{unique fixed point} solution for function $\Q(\cdot)$ satisfying~\eqref{eq:cost2go_periodic}. Such existence can be asserted using the \textit{Banach} fixed-point theorem. We summarize this result next, and readers are referred to~\cite{shapiro2019stationary} for a detailed discussion. To that end, consider the following:
\begin{itemize}
  \item $\mathfrak{B} := \mathbb{B}(\X)$: the Banach space of bounded functions $\mathfrak{g}:=g(x)$, and $g: \X \to \mathbb{R}$ is equipped with the sup-norm $\|g\|_\mathfrak{B} = \sup_{x \in \X}|g(x)|$.
  \item A mapping $\mathfrak{T}: \mathfrak{B} \to \mathfrak{B}$ defined as $\mathfrak{T}(\mathfrak{g})(x) := \mathbb{E}[\psi_{g}(x,\xi)]$ where
 \begin{equation}
  \label{eq:optimality_eq}
  \psi_{g}(x,\xi) := \min_{x' \in \X(x,\xi)}\left\{f(x', \xi) + \gamma g(x')\right\}.
  \end{equation}
\end{itemize}
\begin{prop}
  \label{prop:contraction} (Contraction mapping) The mapping $\mathfrak{T}$ is a contraction mapping, that is, for every $\mathfrak{g}, \mathfrak{g}' \in \mathfrak{B}$ the following inequality holds:
  \begin{equation}
 \label{eq:contraction}
 \|\mathfrak{T}(\mathfrak{g}) - \mathfrak{T}(\mathfrak{g}')\|_{\mathfrak{B}} \leq \gamma \|\mathfrak{g} - \mathfrak{g}'\|_{\mathfrak{B}}.
  \end{equation}
\end{prop}
\begin{thm}
\label{thm:uniqueness}
\cite{shapiro2019stationary} Under Assumptions~\ref{assm:relative_recourse}, \ref{assm:boundedness} and \ref{assm:stationarity}, the following results hold:
\begin{itemize}
  \item Uniqueness: there exists a unique function $Q(\cdot,\cdot)$ satisfying the DP formulation~\eqref{eq:cost2go_periodic}.
  \item Convergence: for any $\mathfrak{g}^0 \in \mathfrak{B}$ the sequence of functions $\{\mathfrak{g}^k\}_{k=0}^{\infty}$ obtained iteratively by
  \begin{equation}
  \label{eq:cost2go_m1}
  \mathfrak{g}^{k}(x)=\mathfrak{T}(\mathfrak{g}^{k-1})(x) = \E[\psi_{g^{k-1}}(x,\xi)] := \E\Big[
  \min_{x^{k} \in \X(x,\xi)}\Big\{f(x^{k},\xi) + \gamma g^{k-1}(x^{k})\Big\}\Big] \quad \forall \; x \in \X,
  \end{equation}

  converges (in the norm $\|\cdot\|_\mathfrak{B}$) to $\Q(\cdot)$.
  \item Convexity: if set $\X$ is convex and function $f(x, \xi)$ is convex in $x \in \X$, then function $\Q:\X \to \mathbb{R}$, is also convex.
\end{itemize}
\end{thm}
The uniqueness and convergence assertions of Theorem~\ref{thm:uniqueness} imply that the unique solution $\mathfrak{g}^{*}=\mathfrak{T}(\mathfrak{g}^{*})$ obtained iteratively using~\eqref{eq:cost2go_m1} is the optimal solution to the DP formulation~\eqref{eq:cost2go_periodic}. Moreover, under the convexity assertion, \cite{shapiro2019stationary} show that $\Q(x)$ is piecewise linear and convex with respect to $x$, and provide an SDDP-type cutting-plane algorithm which gives an optimal policy $\pi_{\infty}$ to the corresponding infinite horizon problem~\eqref{eq:expec_tot_disc_cost}. We refer the reader to~\cite{shapiro2019stationary} for the proofs of Proposition~\ref{prop:contraction} and Theorem~\ref{thm:uniqueness}, a detailed description of the SDDP-type algorithm, and further discussions on the topic.

\section{An Error Bound for Rolling-horizon Policies in Multistage Stochastic
Programming}
\label{sec:error_bound} 
The RH procedure is a well-known approach to construct online decision policies for sequential decision making. In this section, our goal is to provide an error bound to the suboptimality of the online policy corresponding to the RH procedure for solving the stationary discounted infinite horizon MSP model~\eqref{eq:cost2go_periodic}. To that end, let us first discuss the RH procedure and clarify what is meant by an online policy.
\subsection{Rolling-horizon procedure: online vs. offline policies}
\label{subsec:online_vs_offline}
To demonstrate what is meant by an online policy, it might be instructive to revisit and analyze the form of what is considered as an \emph{offline} policy first. We consider the offline policy provided by the standard implementation of the SDDP algorithm as mentioned in Subsection~\ref{subsec:SDDP}. Recall that the final output of the SDDP algorithm is a collection of approximate expected cost-to-go functions $(\check \Q_2(\cdot), \dots, \check \Q_T(\cdot))$. This collection of expected cost-to-go functions is what defines the \emph{offline} policy provided by the SDDP algorithm. In other words, an offline policy is a policy induced by the \emph{approximate value functions} obtained via an offline training procedure. The quality of this offline policy can be evaluated by using a set of $\mathcal{K}$ (out-of-sample) scenarios $\{\xi^k\}_{k\in \mathcal{K}}$, with $|\mathcal{K}| \ll |\Xi_1| \times |\Xi_2| \times \dots \times |\Xi_T|$, and $\xi^k = (\xi_2^{k_2}, \ldots, \xi_T^{k_T})$. \exclude{The set of $\mathcal{K}$ scenarios can then be used to compute the values $z(\xi^k) = \sum_{t=1}^T f_t(\check x_t,\xi^{k_t}_t), \; \forall \; k \in \mathcal{K},$, as well as
\begin{equation}
\label{eq:upper_bound}
%\left\{
\begin{array}{llll}
\tilde  z= \frac{1}{|\mathcal{K}|}\displaystyle\sum_{k \in \mathcal{K}}z(\xi^k), \\
\sigma^2=\frac{1}{|\mathcal{K}|}\displaystyle\sum_{k \in \mathcal{K}}\left[z(\xi^k)-\tilde z\right]^2,
\end{array}
%\right 
\end{equation}
with $\tilde{z}$ and $\sigma^2$ being the sample average and sample variance of $z(\cdot)$, respectively. The sample average $\tilde{z}$ provides an unbiased estimator for an upper bound of the optimal
\begin{equation}
\label{eq:unbaised_ub}
\hat z=\E\left[\sum_{t=1}^T f_t(x_t,\xi_t).\right].
\end{equation}
}

Alternatively, one might consider using an RH procedure where an online policy is given by \emph{solving an MSP problem} on the fly for decision stage $t=1, \dots, T$ during an out-of-sample evaluation. To demonstrate how the RH procedure works, let us define a \textit{roll} as the decision stage at which the DM implements their actions. For any trajectory of the stochastic process $\xi^k = (\xi^{k_2}_2, \dots, \xi^{k_T}_T)$, at the beginning of every roll $t = 1, \dots, T$, based on the current (pre-decision) state $x_t$ and the observed realization $\xi^{k_t}_{t}$, the DM defines a new MSP problem of the form~\eqref{eq:MSP} with $\tau \leq T$ being the number of look-ahead stages. The newly defined problem is then solved (e.g., using the SDDP algorithm) and \emph{only} the first-stage decision $\check x_{t}$ is implemented. The DM then pays the immediate cost $f_t(\check x_{t},\xi^{k_t}_{t})$ and moves forward to the next stage. At that stage ($t+1$), a new problem is defined with state variable $s_t = S_t(\check x_{t},\xi^{k_t}_t)$ being the information carried over from the previous stage, $\xi^{k_{t+1}}_{t+1}$ being the observed realization of the stochastic process, and $\tau \leq T$ (possibly different from the one used in the previous stage) being the number of look-ahead stages. The online policy is then evaluated by collecting all the immediate costs $f_t(\check  x_t,\xi^{k_t}_{t})$ incurred at every roll $t$ in a collection of (out-of-sample) scenarios $\{\xi^k\}_{k\in \mathcal{K}}$. A symbolic representation of the RH procedure is depicted in Figure~\ref{fig:RH-procedure} (in the appendix) and an algorithmic description is provided in Algorithm~\ref{alg:RH-procedure}.

\begin{algorithm}[htbp]
\caption{The rolling horizon procedure (\textit{online} policy).}
\label{alg:RH-procedure}
\begin{algorithmic}[0]
\State \textbf{Input:} Initial value $x_0$ and an out-of-sample path $\xi = (\xi_1, \xi_2, \dots, \xi_T)$.
\State \textbf{STEP 0:} Initialize $z(\xi)=0$ and a vector $\check x := (\check x_0,\check x_1, \dots ,\check x_T)$ with $\check x_0 = x_0$ given as input.
\State \textbf{STEP 1:} For $t=1, 2, \dots, T$
\begin{itemize}
\item Define an MSP problem $\mathcal{P}_t(\tau,\check x_{t-1},\xi_t)$ of the form~\eqref{eq:cost2go}, with $\Q_{t+\tau}(\cdot) := 0$, $\check x_{t-1}$ as the initial value, and $\xi_t$ as the observed realization of the random variable.
\item Solve $\mathcal{P}_t(\tau,\check x_{t-1},\xi_t)$ to obtain the optimal decision $x^*_t$, and set $\check x_{t} = x^*_t$.
\item Update $z(\xi) \to z(\xi) + f_t(x^*_t,\xi_{t})$
\item Set $\check x_{t} = x^*_t$. 
\end{itemize}
\State \textbf{STEP 2:} return $z(\xi)$.
\end{algorithmic}
\end{algorithm}

The MSP problem solved at the beginning of every roll, with $\tau$ being the length of the forecast horizon, can be defined in many different ways. We summarize them in the following three categories and the reader is referred to~\cite{powell2014clearing} for a detailed discussion. The reader is also referred to~\cite{chand2002forecast} for a classified bibliography of the literature on the RH procedure.
\begin{enumerate}
  \item \textbf{Limiting the horizon:} reduce the number of stages in every roll by assuming that $\Q_{t+\tau}(\cdot) := 0$.
  \item \textbf{Stage aggregation:} aggregate the stages from $t' = t+\tau, \dots, T$ in the same spirit as the \emph{partially adaptive} framework introduced in~\cite{zou2018partially} such that
  \begin{footnotesize}
  \begin{equation}
  \label{eq:PAstaget}
  \displaystyle \Q_{t+\tau}(x_{t+\tau-1}) := \mathbb{E}\left[\min_{\substack{x_{t+\tau}, \dots, x_T}}
  \left\{ \sum_{t'=t+\tau}^{T} f_{t'}(x_{t'}, \xi_{t'}) \ \middle\vert
  \begin{array}{l}
  \left(x_{t+\tau},\dots, x_{T}\right) \in \X_{t+\tau}(x_{t+\tau-1},\xi_{t+\tau})\times\dots\times \X_T(x_{T-1},\xi_T),\\
  x_{t'}(\xi_{t'}) = x_{t'},\, \; \forall t'=t+\tau, \dots, T 
  \end{array}\right\}
  \right].
  \end{equation}  
  \end{footnotesize}
  \item \textbf{Outcome aggregation:} use a point estimator $\mu_{t'} = \E[\xi_{t'}]$ for all the exogenous information from $t'=t+\tau, \dots, T$ in the same spirit as the framework discussed in~\cite{pantuso2019number} such that
  \begin{footnotesize}
  \begin{equation}
  \label{eq:OAstaget}
  \displaystyle \Q_{t+\tau}(x_{t+\tau-1}) := \min_{\substack{x_{t+\tau}, \dots, x_T}}
  \left\{ \sum_{t'=t+\tau}^{T} f_{t'}(x_{t'},\mu_{t'}) \ \middle\vert
  \begin{array}{l}
  \left(x_{t+\tau}, \dots, x_{T})\right) \in \X_{t+\tau}(x_{t+\tau-1},\xi_{t+\tau})\times\dots\times \X_T(x_{T-1},\xi_{T})
  \end{array}\right\}.
  \end{equation}  
  \end{footnotesize}
 \end{enumerate}
 For simplicity, in this paper, we will only consider the first strategy (limiting the horizon). Nevertheless, we expect that similar results in this paper can be extended to the other strategies as well. 
 
\subsection{An error bound for RH online policies used for solving stationary infinite horizon discounted models}
\label{subsec:error_bound}
In this subsection, our goal is to construct an error bound to the suboptimality of the RH procedure when used to solve problem~\eqref{eq:cost2go_periodic}. The bound can then be used to derive a \textit{sufficient} forecast horizon $\tau^*_{\epsilon}$ (see Definition~\ref{defn:sufficient_look-ahead}). Our work is inspired by that of~\cite{hernandez1990error} where the authors provide similar ideas for  discrete-time Markov control processes. an

To that end, let us first make clear the definition of a sufficient forecast horizon $\tau^*$. Let $\pi_\tau:= \{x^{\tau}_t\}_{t=1}^{\infty}$ be the sequence of optimal first-stage decisions obtained by an online RH procedure, where $\tau$ is the forecast horizon used in every roll $t=1, \dots, \infty$, such that,
\begin{equation}
\label{eq:total_cost}
  V(\pi_\tau) := f_1(x^{\tau}_1, \xi_t) + \lim_{T\to \infty} \E\left[\sum_{t=2}^{T} \gamma^{t-1} f_t(x^{\tau}_t, \xi_t)\right]
\end{equation}
is the total expected discounted cost when evaluated using $\pi_\tau$. Similarly, we have that $\mathfrak{g}^* = V(\pi_\infty)$ is the minimizer of~\eqref{eq:expec_tot_disc_cost} where 
\begin{equation}
\label{eq:total_cost_optimal}
  \pi_\infty \in \argmin_{\pi \in \Pi} \left\{ f_1(x^{\infty}_1,\xi_t) + \lim_{T\to \infty} \E\left[\sum_{t=2}^{T} \gamma^{t-1} f_t(x^{\infty}_t, \xi_t)\right]\right\}. 
\end{equation}
\begin{defn}
\label{defn:sufficient_look-ahead}
(\emph{$\epsilon$-Sufficient forecast horizon $\tau^*_\epsilon$}) For $\epsilon > 0$, we define a forecast horizon $\tau^*_\epsilon$ as $\epsilon$-sufficient if $V(\pi_{\tau})-V(\pi_\infty) \leq \epsilon,$ for all $\tau \in \{\tau^*, \tau^*+1, \dots, \infty\}$.
\end{defn}
Under the contraction mapping assertion of Proposition~\ref{prop:contraction}, we have the following lemma.
\begin{lem}
\label{lem:contraction} For any $\mathfrak{g}^{0} \in \mathfrak{B}$ and $k=1, \dots, \infty$, the value functions $\mathfrak{g}^{k}$ obtained iteratively using~\eqref{eq:cost2go_m1} satisfies $ \|\mathfrak{g}^{*}-\mathfrak{g}^{k}\| \leq \gamma^k\|\mathfrak{g}^{*}-\mathfrak{g}^{0}\|$. 
\end{lem}
\begin{proof}
$\|\mathfrak{g}^{*}-\mathfrak{g}^{k}\| = \|\mathfrak{T}\mathfrak{g}^{*}-\mathfrak{T}\mathfrak{g}^{k-1}\|=\|\mathfrak{T}^k\mathfrak{g}^{*}-\mathfrak{T}^k\mathfrak{g}^{0}\| \leq \gamma^k\|\mathfrak{g}^{*}-\mathfrak{g}^{0}\|$. 
\end{proof}

Let us define $\mathfrak{g}^{0} = 0$, and suppose we solve~\eqref{eq:cost2go_m1} recursively for $k=\tau$ iterations. We claim that $\mathfrak{g}^{\tau}$ provides the same optimal first-stage action obtained by the RH policy $\pi_\tau$. To see why this is the case, note that since $\mathfrak{g}^{0} = 0$, function $\mathfrak{g}^{1}$ is equivalent to the terminal stage value function $\Q_\tau(\cdot)$ obtained using the policy $\pi_\tau$. Similarly, since $\mathfrak{g}^{2}$ is the minimizer of $\E[f(x^2,\xi)+\gamma g^{1}(x^2)]$, we have that $\mathfrak{g}^{2}$ is equivalent to $\Q_{\tau-1}(\cdot)$. We can then use the same argument by doing backward induction to show the equivalence between $\mathfrak{g}^{k+1}$ and $\Q_{\tau-k}(\cdot)$ for $k=2, \dots, \tau-1$. This concludes that for a given initial state values $x_0$ and $\xi_1$, function $\mathfrak{g}^{\tau}$ is equivalent to the optimal first stage solution of problem~\eqref{eq:1stagepbm}. Moreover, the minimizer given by
\begin{equation}
  \label{eq:RH_tau_policy}
  \displaystyle x^\tau \in \argmin_{x'\in \X(x,\xi)}\left\{f(x', \xi) + \gamma g^{\tau-1}(x')\right\}
\end{equation}
is also a minimizer of~\eqref{eq:1stagepbm} for a given $x \in \X$ and $\xi \in \Xi$. 
\begin{thm}
\label{thm:bound}
Let $\gamma \in (0,1)$, suppose that $V(\pi_\infty)$ is the minimizer of~\eqref{eq:expec_tot_disc_cost} and $V(\pi_\tau)$ is given by~\eqref{eq:total_cost}, then under Assumption~\ref{assm:boundedness}, we have the following: 
\begin{enumerate}[label=(\alph*)]
  \item If $f(x, \xi) \leq 0$, then $V(\pi_\tau)-V(\pi_\infty)\leq \gamma^\tau\frac{\kappa}{1-\gamma}$.
  \item For any general $f(x, \xi): \X \times \Xi \to \mathbb{R}$, we have that $V(\pi_\tau)-V(\pi_\infty)\leq 2 \gamma^\tau \frac{\kappa}{1-\gamma}$.
\end{enumerate}
\end{thm}
\begin{proof}
\begin{enumerate}[label=(\alph*)]
  \item Since $x^\tau$ is the minimizer of~\eqref{eq:RH_tau_policy}, we have that
  \begin{equation}
 \label{eq:RH_tau_eval}
 \displaystyle \mathfrak{g}^{\tau}(x) = \E\left[f(x^\tau, \xi) + \gamma \mathfrak{g}^{\tau-1}(x^{\tau})\right], \ \forall \; x \in \X.
  \end{equation}
  Moreover, since $f(x,\xi) \leq 0$, we have that
  \begin{equation*}
  \mathfrak{g}^{\tau}(x)  \geq  \E\left[\sum_{t=1}^{k} \gamma^{t-1} f(x^\tau, \xi) + \gamma^{k}\mathfrak{g}^{\tau-1}(x^{\tau})\right], \ \forall \; x \in \X, \; k=1, 2, \dots. 
  \end{equation*}
  By letting $k \to \infty$, we get that 
  \begin{equation*}
  \mathfrak{g}^{\tau}(x) \geq \E\left[\sum_{t=1}^{\infty} \gamma^{t-1} f(x^\tau, \xi)\right] = V(\pi_\tau), \ \forall \; x \in \X .\\
  \end{equation*}
Therefore,
\begin{equation*}
 V(\pi_\tau) - V(\pi_\infty) =  V(\pi_\tau)-\mathfrak{g}^{*} \leq \mathfrak{g}^{\tau}-\mathfrak{g}^{*} \leq \gamma^\tau \|\mathfrak{g}^{*} - \mathfrak{g}^{0}\| \leq \gamma^\tau \|\sum_{t=1}^{\infty} \gamma^{t-1} f(x^\tau, \xi)\| \leq \gamma^\tau\frac{\kappa}{1-\gamma}.\\
\end{equation*}
\item Let us define another (immediate) cost function $\bar f(x, \xi) = f(x, \xi) - \kappa$. Since this newly defined function $-2\kappa \leq \bar f(x, \xi) \leq 0$, the result follows from (a) (by redefining $\kappa = 2\times \kappa$).
\end{enumerate}
\end{proof}
Theorem~\ref{thm:bound} provides a bound on the total expected discounted cost obtained by using the RH procedure as a function of the length of the forecast horizon $\tau$. This bound can be used to obtain the following formula for an $\epsilon$-sufficient forecast horizon $\tau^*_\epsilon$ 
\begin{equation}
  \label{eq:sufficient_tau}
  \tau^*_\epsilon = \frac{ \log{(\frac{\epsilon(1-\gamma)}{\kappa})}}{\log{(\gamma)}},
\end{equation}
which guarantees that, for all $\tau \geq \tau^*_\epsilon$, the suboptimality of the respective policy $\pi_{\tau}$ is within an optimality gap of $\epsilon$. Note that, similar bounds can also be constructed for the situation where the number of periods $m>1$. However, since the mapping $\mathfrak{g}$ is going to be defined for all the stages within a period $m$, the bound will scale by a factor $m$ for every increment in the forecast horizon $\tau$. 

While Theorem~\ref{thm:bound} does provide a cue on how to choose the number of look-ahead stages in the RH policies to perform in MSP problems, the formula for an $\epsilon$-sufficient forecast horizon $\tau^*_\epsilon$ provided in~\eqref{eq:sufficient_tau} indicates that
$\tau^*_\epsilon$ could grow quickly as $\gamma$ gets bigger and/or $\epsilon$ gets smaller. From a computational perspective, this is somewhat pathological and defeats the purpose of using an RH procedure -- now that constructing a policy $\pi_{\tau^*_\epsilon}$ with a relatively small $\epsilon$ and large $\gamma$ entails solving a sequence of MSP problems, each with a significantly large horizon $\tau^*$. Another concern with this bound is that its dependency on the original MSP problem is only through the (constant valued) upper bound $\kappa$ on the immediate cost function $f_t(\cdot,\cdot)$. This makes the choice of the employed $\epsilon$-sufficient forecast horizon to be \textit{static}, losing the opportunity to exploit the system state related information. In Section~\ref{sec:ADP_approach}, we exactly address these concerns by constructing a \textit{dynamic} policy which maps the state of the system $s_t$ to a sufficient forecast horizon $\tau^*_t,\; \forall t=1,\dots, T$. To that end, henceforth, we will introduce a subscript $t$ to distinguish between the static forecast horizon $\tau$ and the dynamic forecast horizon $\tau_t$ which is time dependant (via the state $s_t$).

\section{An Approximate Dynamic Programming Approach for Dynamic Look-ahead Policies}
\label{sec:ADP_approach}
In this section, we propose a solution approach for constructing a dynamic RH look-ahead policy from an ADP perspective. Although novel in stochastic programming, similar ideas to our proposed \textit{dynamic} look-ahead policy have been discussed in the MDP literature. For instance,~\cite{bes1986line} construct a dynamic look-ahead policy for the infinite horizon discounted MDP as follows. First, the authors show the existence of a finite sufficient forecast horizon $\tau^*$ for any given system state. Second, they identify the optimal first-stage decision $x^{\tau^*}$ which corresponds to this $\tau^*$. Finally, they propose an optimality criterion to this $x^{\tau^*}$, and then use this result to provide a procedure for detecting the smallest $\tau \in \mathbb{Z}_+$ such that $V(\pi_{\tau})= V(\pi_{\tau^*})$.

The high-level idea of the procedure described in~\cite{bes1986line} for detecting $\tau^*$ is to go through \textit{every} possible system state and \textit{increase} the length of the forecast horizon one stage at a time until $\tau^*$ is identified. To do this, for each step, the algorithm loops over \textit{every} feasible action, and solves a DP problem, where all actions which do not satisfy the (predetermined) optimality criterion for $x^{\tau^*}$ are eliminated. This process is repeated until the finite action space is reduced to a single element, in which case the corresponding forecast horizon is considered sufficient. Similar procedures have also been discussed in~\cite{hernandez1988forecast} for a more general class of MDPs.

Showing the existence of $\tau^*$ and finding the corresponding $x^{\tau^*}$ are somewhat readily accessible for the MDP setting because of the finiteness assumptions on the state/action space. While it seems challenging to extend this procedure in MSP problems with continuous state/action space, we expect that this idea and its featured theoretical results are extendable to MSP problems with pure discrete variables. However, even if such extensions are possible, since this procedure solves a DP for \textit{every} state, action, and $\tau\leq\tau^*$, the computation might scale poorly with the size of the state and action space. Instead, we propose a more practical approach next. 

\exclude{Ultimately, our goal is to construct a function which maps every state $s_t = S_t(x,\xi),\; \forall (x,\xi)\in \X\times\Xi$ at the beginning of every roll $t$ to the smallest forecast horizon $\tau^*_t$ which minimizes the total expected discounted cost. An explicit characterization of such function, therefore, entails solving the following stochastic dynamic optimal control problem:
\begin{equation*}
\min_{\tau_{1}\in\mathbb{N}}f(x_{1}^{\tau_{1,n}}, \xi_{1})+\gamma\E_{\xi_{2}} \left[\rule{0cm}{0.8cm}
\min_{\tau_{2}\in\mathbb{N}}f(x_{2}^{\tau_2}, \xi_{2})+\gamma\E_{\xi_{3}}\left[\rule{0cm}{0.6cm} 
\min_{\tau_{3} \in \mathbb{N}} f(x_{3}^{\tau_3}, \xi_{3})+\cdots \right]\right].
\end{equation*}
This problem, however, has a nested structure -- and more importantly, every decision $x_{t}^{\tau_t} \in \X_t,\; \forall \tau_t \in \mathbb{N}$ and $t=1, 2, \dots$, is itself an optimal solution to the MSP problem described in~\eqref{eq:RH_tau_policy}. This clearly makes the computation impractical, if not impossible. To that end, we consider the following ADP approach.}

\subsection{An approximate dynamic programming approach}
\label{subsec:ADP_approach}
Ultimately, our goal is to construct a function which maps every state $s_t = S_t(x_t,\xi_t),\; \forall (x_t,\xi_t)\in \X_t\times\Xi_t$ at the beginning of every roll $t$ to the smallest sufficient forecast horizon $\tau^*_t$. As previously noted, when using the RH procedure, at the beginning of every roll $t=1, 2, \dots$, the DM solves an MSP problem with a given forecast horizon, implements the first-stage \emph{only}, and repeats this process at the next roll $t+1$. As such, if the inclusion of additional stages in the forecast horizon (i.e., using $\tau'_t> \tau_t$ stages instead of $\tau_t$) does not influence the first-stage decision $x_{t}^{\tau_t}$ (i.e., $x_{t}^{\tau'_t}$ remains the same for all $\tau'_t\geq \tau_t$) then $\tau_t$ can be viewed \textit{sufficient} for this roll $t$. The main idea of our proposed approach is to train the dependency of a ``small and sufficient'' forecast horizon $\tau_t^*$ on the state using an offline training procedure. Specifically, we consider a regression model by treating the system states as the independent variables and treating the corresponding smallest sufficient forecast horizon as the dependent variables in the regression. First, we generate a random sample of the system states $\{s_n\}_{n=1}^{N}$. Then, for each sample $s_n$, we solve an MSP problem using different forecast horizons $\tau \in \{1,2, \dots, \tau^{\max}\}$, and we record the smallest forecast horizon $\tau^*$ for which the optimal first-stage solution $x^{\tau^*}_{1}$ ``converges'', i.e., it remains the same for all $\tau \geq \tau^*$. To practically determine whether or not $x^{\tau}_{1}$ will remain the same for all $\tau \geq \tau^*$ we use the following \emph{stability test}: we increase the length of the forecast horizon $\tau$ one step at a time, and we keep track of the optimal first-stage solution $x^{\tau}_{1}$ for every $\tau\in\{1,2, \dots, \tau^{\max}\}$. Then, whenever $\|x^{\tau}_{1}-x^{\tau-w}_{1}\|<\epsilon$ for some user-specified tolerance parameter $\epsilon > 0$ and stability parameter $w\in \mathbb{N}$, we conclude that $\tau_n^* = \tau-w$. Once all the data points $\{(s_n,\tau_{n}^*)\}_{n=1}^{N}$ has been collected, we use a real-valued function $\Tau : \X\times\Xi \to \mathbb{R}$ to fit these points. To do this, we define a set of \textit{basis functions} $\phi_0(\cdot), \dots, \phi_P(\cdot)$, and look for $\Tau(\cdot)$ in the finite dimensional space spanned by these functions. In other words, we look for $\Tau(\cdot)$ in the form of the following linear combination:
\begin{equation}
  \label{eq:basis_function}
  \Tau(s_t) := \sum_{i=0}^{P} \theta_i\phi_i(s_t),
\end{equation}
and fit a regression model to find the parameters $\theta_0, \dots, \theta_P$. We note that to find a good design of the basis functions, one needs to take advantage of specific problem structures. Given $\theta_0, \dots, \theta_P$ and the form of the basis functions, one can obtain the forecast horizon to consider for each state encountered online during the RH procedure. This approach is summarized in Algorithm~\ref{alg:ADP}.
 \begin{algorithm}[htbp]
\caption{An offline learning approach for dynamic selection of forecast horizons}
\label{alg:ADP}
\begin{algorithmic}[0]
\State \textbf{Input:} A sample size $N$, a sample $\{s_n\}_{n=1}^{N}$, a tolerance parameter $\epsilon > 0$, a stalling parameter $w \in \mathbb{N}$, and a maximum forecast horizon $\tau^{\max}<T$. 
\State \textbf{Step 0:} Initialize $n = 0$, and a set of candidate sufficient forecast horizons (one for each sample) $\{\tau^*_n\}_{n=1}^{N}$.
\State \textbf{Step 1:} Let $n \gets n+1$.
\State \textbf{Step 2:} Initialize a set of candidate optimal first-stage solutions $X^{*}_n = \emptyset$ and a candidate forecast horizon $\tau\in\{1,2, \dots, \tau^{\max}\}$.
\begin{itemize}
  \item \textbf{while} {\emph{true}} \textbf{do}
  \begin{enumerate}[label=(\alph*)]
 \item Given $s_n$ as an input, solve an MSP problem of the form~\eqref{eq:cost2go} with $T = \tau$ to obtain the optimal (first-stage) action $x^{\tau}_{1}$ and append $x^{\tau}_{1}$ to $X^{*}_n$.
 \item If $|X^{*}_n| > w$ go to (c). Otherwise, go to (d).
 \item If $\|x^{\tau}_{1}-x^{\tau-w}_{1}\|<\epsilon$, set $\tau_n^* = \tau-w$ and \textbf{stop}. Else, go to (d).
 \item If $\tau < \tau^{\max}$, let $\tau\gets \tau+1$ and go to (a). Else, set $\tau_{n}^* = \tau^{\max}$ and \textbf{stop}.
 \end{enumerate}
\end{itemize}
\State \textbf{Step 3:} If $n=N$, \textbf{return} $\{(s_n,\tau^*_n)\}_{n=1}^{N}$ and \textbf{stop}. Else, go to \textbf{Step 1}.
\State \textbf{Step 4:} Define $\Tau(s_t) := \sum_{i=0}^{P} \theta_i\phi_i(s_t)$ and fit a regression model to find $\theta_0, \dots, \theta_P$ using $\{(s_n,\tau^*_n)\}_{n=1}^{N}$.
\end{algorithmic}
\end{algorithm}

Next, we illustrate the proposed approach on a class of \textit{hydrothermal energy planning} problems, which are used as the benchmark instances in our numerical experiments shown in Section~\ref{sec:numerical_results}.

\subsection{An illustrating example: the hydrothermal power operation planning problem}
\label{subsec:HPOP}
The hydrothermal power generation system is one of the main energy sources in many countries. For example, the Brazilian hydrothermal power system is a large scale network of facilities that can be used to produce and distribute energy by circulating H$_2$O fluids (water). The power plants in the Brazilian network can be categorized into two types:
\begin{enumerate}
\item A set of hydro plants $H$, which has no production cost, but for each hydro plant $h \in H$ there is an upper limit $\bar q_h$, which is the maximum allowed amount of turbined flow that can be used for power generation. These hydro plants can also be categorized further into two types:
\begin{enumerate}
\item Hydro plants with reservoirs $H_R$, and for each $h \in H_R$ there is an upper and lower limit on the level of water allowed in the reservoirs, denoted by $\underline{v}_{h, t}$ and $\bar v_{h, t}$, respectively.
\item Hydro plants without reservoirs (run-of-river) $H_I$, and thus no water storage is possible.
\end{enumerate}
Moreover, for each hydro plant $h$, a unit of released water will generate $r_h$ units of power, also known as efficiency rate. The set of immediate upper and lower stream plants for hydro plant $h$ in the network is given by $U(h)$ and $L(h)$, respectively.
\item A set of thermal plants $F$, where for each thermal plant $f \in F$, the minimum and maximum amount of  power generation allowed is $\underline{g}_{f, t}$ and $\bar g_{f, t}$, respectively, $\forall \; t =1, \dots, T$.  Additionally, each plant $f \in F$ is associated with an operating cost of $c_{f,t}$ for each unit of thermal power generated. We denote the cost vector of generating thermal power at time $t$ by $\vec{c}_{g, t}$.
\end{enumerate}

In the hydrothermal power operation planning problem (HPOP) problem, the goal of the DM is to find an operation strategy $\pi$, in order to meet production targets (demands) $d_{t}$ for each decision stage $t, \; \forall \; t=1, \dots, T$, while minimizing the overall production cost. To do this, at each decision stage $t$ and given the state of the system $s_t$ (i.e., the water level at each hydro plant $h \in H_R$), the DM has to decide on how much water to turbine by each hydro plant $h \in H$ and how much thermal power to generate by each plant $f \in F$. In the event where the levels of energy production using hydro/thermal plants does not meet the levels of demand, the DM must pay a penalty of $c_{p,t}$ for each unit of unsatisfied demand. Such penalty can also be thought of as a cost of buying energy from the outside market, which typically has a higher cost compared to the cost of generating energy using internal resources (thermal plants), i.e., $c_{p,t}  > \max_{f\in F}\{c_{f,t}\}$. 

The HPOP problem can be modeled as an MSP problem because of the uncertainty in the problem data, such as future inflows (amount of rainfall) $\tilde b_t := \vec{\tilde b}_{t} (\xi_t)$, demand $d_t := \vec{\tilde d}_{t}(\xi_t)$, fuel costs $\tilde c_t := \vec{\tilde c}_{t}(\xi_t)$ and equipment availability~\cite{shapiro2011report}. We will use the notation $b_{t}, d_{t}, c_{t}$ whenever the parameter is deterministic (e.g.,  stage $t=1$) and $\tilde b_{t}, \tilde d_{t}, \tilde c_{t}$ whenever it is random. To introduce a mathematical programming formulation for the decision problem which the DM has to solve at every decision stage $t$, for $t=1, \dots, T$, consider the following decision variables:
\begin{itemize}
\item $\vec{x}_t \in \mathbb{R}^{H_R}_{+}$: where $x_{h, t}$ is the amount of water stored at each hydro plant with a reservoir $h \in H_R$.
\item $\vec{y}_t \in \mathbb{R}^{H}_{+}$: where $y_{h, t}$ is the amount of water turbined by each hydro plant $h \in H$.
\item $\vec{g}_t \in \mathbb{R}^{F}_{+}$: where $g_{f, t}$ is the amount of thermal power generated by each thermal plant $f \in F$.
\item $\vec{v}^{+}_t, \vec{v}^{-}_t \in \mathbb{R}^{H_R}_{+}$: where $v^{+}_{h, t}, v^{-}_{h, t}$ is the amount of spilled/pumped-back water (without generating power) in hydro plant $h \in H_R$
\item $p_t \in \mathbb{R}_{+}$: the amount of unsatisfied demand at stage $t$. 
\end{itemize}
See Table~\ref{tab:HPOP_notation} in the appendix for a summary of all the notations used in the HPOP problem. The local optimization problem at every decision stage $t$, for $t = 1, \dots, T$ is defined as follows:
\begin{subequations}\label{eq:HPOP}
\begin{footnotesize}
\begin{align}
\min_{\vec{x}_t,\vec{y}_t,\vec{g}_t, \vec{v}^{+}_t, \vec{v}^{-}_t, p_t} \quad & \vec{\tilde c}_{g, t}^{\intercal} \vec{g}_t + \tilde c_{p, t} p_t\\
\textrm{s.t.} \quad & x_{h, t} = x_{h,t-1} + \tilde b_{h, t} + \Bigg[\sum_{m \in U(h)} (y_{m, t} + v^{+}_{m, t})-\sum_{m \in L(h)} v^{-}_{m, t} \Bigg] - (y_{h, t} + v^{+}_{h, t}-v^{-}_{h, t}), \quad \forall \; h \in H_R \label{eq:HPOP-1} \\
  &  y_{h, t} + v^{+}_{h, t}-v^{-}_{h, t} = \tilde b_{h, t} + \Bigg[\sum_{m \in U(h)} (y_{m, t} + v^{+}_{m, t})-\sum_{m \in L(h)} v^{-}_{m, t} \Bigg],\quad \forall \; h \in H_I \label{eq:HPOP-2} \\
  & \sum_{h \in H} y_{h, t}r_h + \sum_{f\in F} g_{f, t} + p_t \geq d_t \label{eq:HPOP-3} \\
  & \underline{v}_{h, t} \leq x_{h, t} \leq \bar v_{h, t}, \quad \; \forall h \in H_R \label{eq:HPOP-4} \\
  & y_{h, t} \leq \bar q_h, \quad \forall \; h \in H \label{eq:HPOP-5} \\
  & \underline{g}_{f, t} \leq g_{f, t} \leq \bar g_{f, t}, \quad \; \forall f \in F \label{eq:HPOP-6} \\
  & \vec{x}_t,\vec{y}_t,\vec{g}_t, \vec{v}^{+}_t, \vec{v}^{-}_t, p_t \geq 0. \label{eq:HPOP-7}
\end{align}
\end{footnotesize}
\end{subequations}

Constraints~\eqref{eq:HPOP-1} describe the flow balance in the system, which gives the (post-decision) water level $x_{h,t}$ stored in a reservoir $h\in H_{R}$ at the end of stage $t$. Constraints~\eqref{eq:HPOP-2} describe the system dynamics in the run-of-river and it is similar to the first set except that the state variables representing the water level in the reservoirs are excluded. Constraint~\eqref{eq:HPOP-3} models the demand requirement. Constraints~\eqref{eq:HPOP-4} to constraints~\eqref{eq:HPOP-6} are the capacity constraints for the reservoirs, the amount of water that can be turbined, and the thermal power generation, respectively. Note that in this context, the system state is given by $s_{h,t} = S_t(x_{h,t-1},\tilde b_{h,t})= x_{h,t-1} + \tilde b_{h,t}$ for a hydro plant with a reservoir $h \in H_{R}$ and $s_{h,t} = \tilde b_{h,t}$ for each run-of-river hydro plant $h \in H_{I}$.

In our implementation, to construct the mapping $\Tau(\cdot)$ from the system state to a sufficient forecast horizon for this HPOP problem, we consider a piecewise linear function where each piece is given by~\eqref{eq:basis_function} with $P=1$. The basis functions for each piece are defined as: $\phi_0(\vec{s}_t) = 1$ (following the ADP literature for numerical stability, see, e.g., \cite{nasrollahzadeh2018real}) and $\phi_1(\vec{s}_t) = \sum_{h\in H} r_h s_{h,t}$. In other words, we define $\phi_1(\vec{s}_t)$ as the total amount of electricity that can be generated by all the hydro plants given their current levels of water in the reservoir. As such, performing the regression now reduces to estimating parameters $\theta_0$ and $\theta_1$ for each piece of the piecewise linear function through an offline training procedure. 

\section{Numerical Experiments}
\label{sec:numerical_results}
In this section, we present our computational results on the empirical performance of different RH policies. In Section~\ref{subsec:implementation} we provide an overview of the implementation details of the different RH policies used in our numerical experiments. We then describe the test instances in Section~\ref{subsec:test_data}. Afterwards, we compare the performance of various approaches considered in the experiment in Section~\ref{subsec:results}. Finally, we present some sensitivity analysis on some key parameters used in the proposed algorithms in Section~\ref{subsec:sensitivity_analysis}.

\subsection{Implementation details}
\label{subsec:implementation}
We implemented two different policies for choosing the forecast horizon in the RH procedure: a \emph{static} policy where the forecast horizon is kept the same in each roll, and a \emph{dynamic} policy where the forecast horizon is dynamically chosen based on the system state according to the regression model trained offline as described in Algorithm~\ref{alg:ADP}. To train this function, we use a sample size of $N = 50$, a tolerance parameter $\epsilon = 0.001\%$, a stalling parameter $w=10$, and a maximum forecast horizon $\tau^{\max} = 64$ (we chose $64$ as the maximum forecast horizon due to limited computational budget and the observation that the policy performance does not improve significantly beyond that). We refer to this procedure as the \emph{offline} training step. We then evaluate both the static and the dynamic RH policies on a sample path with $T = 5000$ stages as the out-of-sample evaluation. We refer to this step as the \emph{online} evaluation step. To measure the performance of each policy, we calculate the long-run average cost which is given by
\begin{equation}
\label{eq:long_run_avg}
\bar z = \frac{1}{T}\displaystyle \sum_{t=1}^{T} f_t(x^{\tau_t}_{t},\xi_t).
\end{equation}

It might seem computationally prohibitive to consider evaluating the out-of-sample performance on a sample path of $T = 5000$ stages for different policies to choose the number of look-ahead stages in each roll, static or dynamic. However, the situation simplifies dramatically under our stationarity assumption (see Assumption~\ref{assm:stationarity}) where the expected cost-to-go functions $\Q_t(\cdot)$ for the MSPs defined in every roll are identical for any fixed length of the forecast horizon. In this case, one can then \emph{reuse} the approximate value function obtained in previous rolls in the current roll as an initial approximate value function, instead of redefining a new MSP problem and training the value function starting from scratch. Since the problems solved in previous rolls were solved under initial states that may not be necessarily the same as the one observed in the current roll, the approximate value function inherited from previous rolls might not be a good enough approximation under the new initial state in the current roll, and therefore need to be evaluated and trained further if necessary. Note that an algorithmic procedure similar to that described in Algorithm~\ref{alg:RH-procedure} can be employed here. The only difference is that, as we move forward over time, we do not dispose of the cuts generated in previous rolls. Instead, we carry them over to the next roll and then resolve the problem under the new initial state $s_t$ (if needed). This idea of ``reusing'' the approximate value functions has appeared in the literature (see, e.g.,~\cite{silva2019data}).

We care to remark the following subtle distinction between reusing the approximate value functions in the static policy vs. dynamic policy for choosing the number of look-ahead stages in the RH procedure. In the static policy, a single MSP problem is defined and the same approximate value functions in every stage are simply reused over the course of the entire procedure. Whereas in the dynamic policy, we keep track of multiple MSP problems (and the associated approximate value functions) one for each forecast horizon $\tau_t$. If the forecast horizon $\tau_t$ given by $\Tau(\vec{s}_t)$ at the beginning of every roll has been observed before, we reuse the corresponding (previously defined) MSP problem which has the same $\tau_t$ value. Otherwise, a new MSP problem is defined and is added to the list of MSP problems with the corresponding forecast horizon encountered so far.

In addition, as previously noted, discount factors have many economical interpretations where the DM accounts for the time value of the costs. One such interpretation is that the future cost is considered less valuable than the current cost. Meaning that, the smaller the discount factor, the less the DM values future costs and vice versa. It is intuitively clear that for each discount factor, there exists a certain length of the look-ahead horizon, both of which represent how the DM values future costs. To show this analogy, we implemented the periodic variant of the SDDP algorithm described in~\cite{shapiro2019stationary} using different values for the discount factor $\gamma \in \{0.1, 0.2, \dots, 0.9, 0.95, 0.99\}$ and made a comparison with RH policies. 

In summary, we compare the performance of the following policies:
\begin{enumerate}
  \item A \emph{static} RH policy that reuses the cuts generated during the online evaluation steps with no discount factor.
  \item A \emph{dynamic} RH policy that reuses the cuts generated during the online evaluation steps with no discount factor.
  \item A stationary policy trained with a given discount factor using the periodic variant of the SDDP algorithm described in~\cite{shapiro2019stationary}.
\end{enumerate}

In all of our experiments, we consider varying the number of sample paths used in each iteration of the SDDP algorithm for solving the MSP problems. However, we observe that using a single sample path per forward/backward step works the best, and thus we use it in all of our experiments. We use the following termination criteria for stopping the SDDP algorithm: (i) a maximum number of $10^5$ iterations is achieved; or (ii) a time limit of 10,800 seconds (three hours) is achieved; or (iii) the lower bound does not progress by more than $\epsilon$ (in the relative term) in $\bar{j}\in \mathbb{Z}_+$ consecutive iterations (i.e., $\frac{(\text{LB}^i- \text{LB}^{i-\bar{j}})}{\text{LB}^i}<\epsilon$). We refer to $\bar{j}$ as the ``stalling parameter'', which is set to $\bar{j} = 500$ by default. When solving the MSP problems in an RH procedure during the out-of-sample evaluation, we gradually tune down the online training effort by decreasing this parameter $\bar{j}$ to accelerate computation without sacrificing the solution quality by much. The rationale behind this dynamic parameter setting is that, as the online evaluation procedure proceeds, the state space is explored more extensively, implying that the quality of the approximate value functions keeps improving. Therefore, it does not worth the same effort to do online training in later rolls as initial rolls. Specifically, we gradually tune down the online training effort by doing the following: 
\begin{enumerate}
  \item Starting from the second roll ($t=2$), we set $\bar{j}=50$.
  \item Then, once all of the realizations of the random data process $\xi_t$ are encountered in the online evaluation procedure, we set $\bar{j}=10$.
  \item Finally, if the algorithm keeps getting terminated as soon as the stalling parameter $\bar{j}=10$ is hit (i.e., right after $i=11$ iterations) for more than 50 consecutive rolls, we turn off the online training by setting $\bar{j}=1$.
\end{enumerate}
We remark that all parameters above are chosen in a rather heuristic manner and can be tuned according to specific problem instances to improve the computational performance. 

\subsection{Test instances}\label{subsec:test_data}
All proposed approaches are tested using benchmark instances motivated by the Brazilian HPOP problem described in Section~\ref{subsec:ADP_approach}. To create a variety of instances, we consider different values for the demand parameter $d_t \in \{1000, 1500, 1750, 2000, 2250\}$, number of hydro plants in the network $|H| \in \{1, 3, 6\}$, number of realizations $|\Xi_t| \in \{5, 12\}$ and probability distribution of the random process $\xi_t$ as shown in Table~\ref{tab:5_real_dist} and Table~\ref{tab:12_real_dist} (in the appendix). In these instances, hydro plants that have reservoirs correspond to $H_R = \{1,3,4\}$, and the run-of-river hydro plants are $H_I = \{2,5,6\}$. The problem  parameters used in formulation~\eqref{eq:HPOP} are given in Table~\ref{tab:hydro_data} (in the appendix). Note that in instances where $|H|=1$, we chose the third hydro plant ($h=3$) and in instances where $|H|=3$, we chose the second, third and fourth hydro plants. In addition to the hydro plants, we also consider a set of four thermal plants $F$ with maximum power generation capacity $\bar g_{f,t}$ and cost $c_{f,t}$ for each thermal plant $f \in F$ as shown in Table~\ref{tab:thermo_data} (in the appendix). Finally, we use $c_{p,t} = 500$ as the penalty parameter for each unit of unsatisfied demand and $c_0 = 2.592$ as the parameter for converting the amount of water flow into the water level in the reservoirs.

All algorithms were implemented in \emph{Julia} 1.4.0, using \emph{JuMP} 0.18.4 package \cite{dunning2017jump}, with commercial solver \emph{Gurobi}, version 9.0.0. All of the tests were conducted on Clemson University's primary high-performance computing cluster, the \emph{Palmetto} cluster, where we use an \emph{R830 Dell Intel Xeon} ``big memory'' compute node with 2.60GHz, 1.0 TB memory, and 24 cores.
 
\subsection{Numerical Results}
\label{subsec:results}
The first step in our implementation for the dynamic RH policy is the \emph{offline} training step where we estimate parameters $\theta_0$ and $\theta_1$ in each piece (see~\eqref{eq:basis_function}) of the piecewise linear function. The choice of using a piecewise linear function as the regression function is inspired by our preliminary observations that, in these instances, the relationship between $s^n$ and $\tau^n$ exhibit a piecewise linear structure with up to three pieces. The best fit results are shown in Table~\ref{tab:decomposed_fit}.
%%%%%%%%%%%%%%%%%%%%%%%%%%%%%%%%%%%%%%%%%%%%%%%%%%%%%%%%%
\begin{table}
\begin{footnotesize}
\begin{center}
\begin{tabular}{|c|c|c|ccc|cc|c|c|}
\hline 
$|H|$ & $\tilde d_t$ & $|\Xi_t|$ & \multicolumn{3}{c|}{\textbf{Range}} & $\hat \theta_0$ & $\hat \theta_1$& $R^2$& $R^2_{\text{avg}}$ \\\hline 
\multirow{12}{*}{1} & \multirow{6}{*}{1000} & \multirow{3}{*}{5} & $0\leq$  & $\phi_1(\vec{s}_t)$ & $< 3100$ & -6.69 & $1.00\times 10^{-2}$ & 66.36\% & \multirow{3}{*}{88.61\%} \\
& & & $3100\leq$  & $\phi_1(\vec{s}_t)$ & $< 14900$ & -1.59 & $1.23\times 10^{-3}$ & 99.48\% & \\
& & & $14900\leq$  & $\phi_1(\vec{s}_t) $ &  & 5.00& 0.00 & - & \\ \cline{3-10} 
& & \multirow{3}{*}{12} & $0\leq$  & $\phi_1(\vec{s}_t)$ & $<  3100$ & -14.61  & $2.04\times 10^{-2}$ & 80.06\% & \multirow{3}{*}{93.19\%} \\
& & & $3100\leq$  & $\phi_1(\vec{s}_t)$ & $<  13000$ & -1.95 & $1.41\times 10^{-3}$ & 99.49\% & \\
& & & $13000\leq$  & $\phi_1(\vec{s}_t) $ &  & 4.00& 0.00 & - & \\ \cline{2-10} 
& \multirow{6}{*}{1500} & \multirow{3}{*}{5} & $0\leq$  & $\phi_1(\vec{s}_t)$ & $<  3000$ & -10.61  & $1.27\times 10^{-2}$ & 63.08\% & \multirow{3}{*}{85.57\%} \\
& & &$3000\leq$  & $\phi_1(\vec{s}_t)$ & $<11500$ & -3.86 & $2.00\times 10^{-3}$ & 99.38\% & \\
& & & $11500\leq$  & $\phi_1(\vec{s}_t) $ &  & -0.81 & $6.90\times 10^{-4}$ & 94.24\% & \\ \cline{3-10} 
& & \multirow{3}{*}{12} & $0\leq$  & $\phi_1(\vec{s}_t)$ & $<  2500$& -1.00 & $2.90\times 10^{-3}$ & 87.50\% & \multirow{3}{*}{94.57\%} \\
& & & $2500\leq$  & $\phi_1(\vec{s}_t)$ & $<  9300$ & -5.33 & $2.55\times 10^{-3}$ & 99.00\% & \\
& & & $9300\leq$  & $\phi_1(\vec{s}_t) $ &  & -0.74 & $7.35\times 10^{-4}$ & 97.22\% & \\ \hline 
\multirow{6}{*}{3} & \multirow{4}{*}{1750} & \multirow{2}{*}{5} & $0\leq$  & $\phi_1(\vec{s}_t)$ & $< 8000$ & 5.08 & $1.16\times 10^{-3}$ & 24.95\% & \multirow{2}{*}{56.35\%} \\
& & & $8000 \leq$  & $\phi_1(\vec{s}_t) $ &  & -1.32 & $1.18\times 10^{-3}$ & 87.74\% & \\ \cline{3-10} 
& &\multirow{2}{*}{12} & $3100\leq$  & $\phi_1(\vec{s}_t)$ & $<  13000$ & 0.70 & $1.81\times 10^{-3}$ & 77.27\% & \multirow{2}{*}{84.27\%} \\
& & & $13000\leq$  & $\phi_1(\vec{s}_t) $ &  & 1.07 & $8.44\times 10^{-4}$ & 91.27\% & \\ \cline{2-10} 
& \multirow{2}{*}{2250} & 5& $0\leq$  & $\phi_1(\vec{s}_t) $ &  & 0.49 & $1.04\times 10^{-3}$ & 84.52\% & 84.52\%\\ \cline{3-10} 
& & 12 & $0\leq$  & $\phi_1(\vec{s}_t) $ &  &0.76 & $8.59\times 10^{-4} $ & 83.87\% & 83.87\% \\ \hline  
\multirow{12}{*}{6} & \multirow{6}{*}{2000} & \multirow{3}{*}{5} & $0\leq$  & $\phi_1(\vec{s}_t)$ & $< 4500$ & -0.55  & $4.07\times 10^{-3}$ & 73.96\% & \multirow{3}{*}{91.17\%} \\ 
& & & $4500\leq$  & $\phi_1(\vec{s}_t)$ & $< 13000$ & -5.40  & $2.34\times 10^{-3}$ & 99.53\% & \\ 
& & & $13000\leq$  & $\phi_1(\vec{s}_t) $ &  & -7.21  & $1.47\times 10^{-3}$ & 83.20\% & \\ \cline{3-10} 
& & \multirow{3}{*}{12} & $0\leq$  & $\phi_1(\vec{s}_t)$ & $<  3500$ & 6.19 & $2.80\times 10^{-3}$ & 23.27\% & \multirow{3}{*}{74.22\%} \\
& & & $3500\leq$  & $\phi_1(\vec{s}_t)$ & $<  11500$ & -5.85  & $2.61\times 10^{-3}$ & 99.38\% & \\
& & & $11500\leq$  & $\phi_1(\vec{s}_t) $ &  & -8.03  & $1.60\times 10^{-3}$ & 60.84\% & \\ \cline{2-10} 
& \multirow{6}{*}{2500} & \multirow{3}{*}{5} & $0\leq$  & $\phi_1(\vec{s}_t)$ & $<  4600$ & -6.72  & $4.63\times 10^{-3}$ & 90.78\% & \multirow{3}{*}{90.55\%} \\ 
& & &$4600\leq$  & $\phi_1(\vec{s}_t)$ & $<10100$ & -6.44  & $2.63\times 10^{-3}$ & 99.35\% & \\ 
& & & $10100\leq$  & $\phi_1(\vec{s}_t) $ &  & -13.48 & $2.06\times 10^{-3}$ & 81.53\% & \\ \cline{3-10} 
& & \multirow{3}{*}{12} & $0\leq$  & $\phi_1(\vec{s}_t)$ & $<  4300$ & -14.59 & $7.69\times 10^{-3}$ & 80.35\% & \multirow{3}{*}{86.78\%} \\ 
& & & $4300\leq$  & $\phi_1(\vec{s}_t)$ & $<  9200$ & -8.36  & $3.15\times 10^{-3}$ & 98.45\% & \\ 
& & & $2300\leq$  & $\phi_1(\vec{s}_t) $ &  & -5.61  & $1.53\times 10^{-3}$ & 81.54\% & \\ \hline
\end{tabular}
\end{center}
\caption{The regression fit for the piecewise linear basis function parameters $\hat \theta_0$ and $\hat \theta_1$, and the \textit{(aggregated) coefficient of determination} ($R^2_{avg}$) $R^2$.}
\label{tab:decomposed_fit}
\end{footnotesize}
\end{table}
To measure the quality of the obtained regression functions, we use an \textit{aggregated coefficient of determination} $R^2_{\text{avg}}$, which is given by a weighted summation of the $R^2$ values over all pieces representing the function. For example, in the test instance where $|H|=1, d_t = 2250$ and $|\Xi_t| = 5$, the piecewise linear regression function contains three pieces with an average $R^2_{\text{avg}} = 85.57\%$ (the $R^2$ values are $63.08\%, 99.38\%$, and $94.24\%$ respectively for each piece), which is almost sixteen times higher than the value when we use a simple linear regression (i.e., with a single piece). Similar improvements in terms of the $R^2_{\text{avg}}$ values can be observed for most instances, except two instances where $|H|=3$ and $d_t = 2250$, where the coefficients of determinations given by the simple linear regression are already satisfactory.

To analyze the performance of different RH look-ahead policies and the stationary policy, we record the following performance metrics:
\begin{itemize}
  \item The training \textbf{time} (in \emph{seconds}): which is recorded \emph{online} in the case of the RH policies, and \emph{offline} in the case of the stationary policy obtained by the periodic variant of the SDDP algorithm.
  \item The \textbf{long-run average cost}: $\bar z$ given by~\eqref{eq:long_run_avg} which is denoted by $\bar z_\tau$ in the case of an RH policy with $\tau$ look-ahead stages, and $\bar z_\gamma$ in the case of stationary policy with discount factor $\gamma$. 
  \item The \emph{relative} \textbf{gap} which is given by $(\bar z_\tau-\bar z_{64})/\bar z_{64}\%$ in the case of RH policies and is given by $(\bar z_\gamma-\bar z_{0.99})/\bar z_{0.99}\%$. We treat $\bar{z}_{64}$ as the ``optimal'' RH policy as the length of the look-ahead horizon is large enough in the sense that the policy performance does not improve significantly beyond that, according to our computational experiments. We treat $\bar z_{0.99}$ as the ``optimal'' stationary policy as the discount factor of $0.99$ is close enough to $1$.
\end{itemize}

All of the numerical results for instances where $|H|=1$ are shown in Table~\ref{tab:RH_numerical_results} to Table~\ref{tab:discount_numerical_results_6H}. Next, we summarize our main observations.

\begin{table}[htbp]
\begin{footnotesize}
\begin{center}
\begin{adjustbox}{width=\textwidth}
\begin{tabular}{|c|c|c|ccc|c|ccc|c|ccc|}
\hline
\multirow{2}{*}{$|\Xi_t|$} & \multirow{2}{*}{$\tau$} & \multicolumn{4}{c|}{$|H|=1$}  & \multicolumn{4}{c|}{$|H|=3$}  & \multicolumn{4}{c|}{$|H|=6$} \\ \cline{3-14} 
 & & $d_t$ & \textbf{Time} & \textbf{$\bar z_\tau$} & \textbf{Gap} & $d_t$ & \textbf{Time} & \textbf{$\bar z_\tau$} & \textbf{Gap} & $d_t$ & \textbf{Time} & \textbf{$\bar z_\tau$} & \textbf{Gap} \\ \hline
\multirow{8}{*}{5}   & 1 & \multirow{16}{*}{1000} & 3.65 & 13365.99  & 40.38\% & \multirow{16}{*}{1750} & 3.72 & 74337.5 & 110.71\% & \multirow{16}{*}{2000} & 14.22 & 47455.92  & 269.12\% \\  
 & 2 &  & 5.43 & 11417.63  & 19.92\% &  & 7.69 & 52646.73  & 49.23\% &  & 17.06 & 21520.59  & 67.39\% \\  
 & 4 &  & 18.29 & 10148.72  & 6.59\%  &  & 60.51 & 41181.76  & 16.73\% &  & 131.58 & 16273.22  & 26.58\% \\  
 & 8 &  & 114.5 & 9724.37 & 2.13\%  &  & 509.18 & 38010.33  & 7.74\%  &  & 1524.64  & 14342.81  & 11.56\% \\  
 & 16 &  & 693.13 & 9546.52 & 0.27\%  &  & 2020.42  & 36264.1 & 2.79\%  &  & 6194.63  & 13385.28  & 4.11\% \\  
 & 32 &  & 2305.26  & 9534.69 & 0.14\%  &  & 5641.59  & 35415.53  & 0.38\%  &  & 15972.43 & 12956.27  & 0.78\% \\  
 & 64 &  & 6462.9 & 9521.29 & - &  & 13567.29 & 35279.98  & - &  & 34677.73 & 12856.51  & - \\  
 & dynamic   &  & 333.99 & 9576.75 & 0.58\%  &  & 661.78 & 37755.31  & 7.02\%  &  & 2584.86  & 13199.63  & 2.67\% \\ \cline{1-2} \cline{4-6} \cline{8-10} \cline{12-14} 
\multirow{8}{*}{12}  & 1 &  & 3.71 & 11613.72  & 42.33\% &  & 11.02 & 87358.81  & 96.39\% &  & 14.14 & 40188.54  & 235.38\% \\  
 & 2 &  & 9.24 & 10167.24  & 24.60\% &  & 12.75 & 60440.18  & 35.88\% &  & 32.05 & 19429.49  & 62.14\% \\  
 & 4 &  & 39.78 & 8958.41 & 9.79\%  &  & 257.2 & 51222.79  & 15.16\% &  & 421.47 & 14799.7 & 23.51\% \\  
 & 8 &  & 326.72 & 8440.7  & 3.44\%  &  & 1891.9 & 47664.75  & 7.16\%  &  & 3368.73  & 13268.63  & 10.73\% \\  
 & 16 &  & 1816.15  & 8221.67 & 0.76\%  &  & 5216.76  & 45670.92  & 2.67\%  &  & 9991.14  & 12431.27  & 3.74\% \\  
 & 32 &  & 5566.49  & 8165.87 & 0.08\%  &  & 12876.73 & 44659.88  & 0.40\%  &  & 24691.71 & 12068.17  & 0.71\% \\  
 & 64 &  & 15372.1  & 8159.69 & - &  & 32105.21 & 44481.48  & - &  & 59671.12 & 11982.87  & - \\  
 & dynamic   &  & 1809.69  & 8314.8  & 1.90\%  &  & 715.67 & 48737.93  & 9.57\%  &  & 4816.96  & 12387.48  & 3.38\% \\ \hline
\multirow{8}{*}{5}   & 1 & \multirow{16}{*}{1500} & 3.64 & 65637.12  & 62.16\% & \multirow{16}{*}{2250} & 3.72 & 222334.09 & 19.82\% & \multirow{16}{*}{2250} & 14.93 & 93311.6 & 274.95\% \\  
 & 2 &  & 5.52 & 51512.1 & 27.26\% &  & 7.22 & 194007.97 & 4.55\%  &  & 19.1 & 44648.11  & 79.41\% \\  
 & 4 &  & 24.36 & 44917.91  & 10.97\% &  & 40.67 & 187064.07 & 0.81\%  &  & 140.71 & 32492.63  & 30.56\% \\  
 & 8 &  & 164.29 & 42523.65  & 5.05\%  &  & 278.58 & 185591.37 & 0.02\%  &  & 1441.57  & 28113.77  & 12.97\% \\  
 & 16 &  & 731.55 & 41178.16  & 1.73\%  &  & 931.33 & 185404.47 & -0.08\% &  & 5428.05  & 26128.32  & 4.99\% \\  
 & 32 &  & 2413.37  & 40540.81  & 0.16\%  &  & 2165.64  & 185404.47 & -0.08\% &  & 14387.64 & 25113.81  & 0.91\% \\  
 & 64 &  & 6947.46  & 40477.72  & - &  & 2752.87  & 185560.35 & - &  & 32775.48 & 24886.25  & - \\  
 & dynamic   &  & 497.22 & 41195.23  & 1.77\%  &  & 38.36 & 186375.7  & 0.44\%  &  & 1928.97  & 27090.02  & 8.86\% \\ \cline{1-2} \cline{4-6} \cline{8-10} \cline{12-14} 
\multirow{8}{*}{12}  & 1 &  & 3.75 & 58454.13  & 56.89\% &  & 3.7  & 246632.4  & 6.41\%  &  & 14.16 & 83937.73  & 259.96\% \\  
 & 2 &  & 9.43 & 45440.69  & 21.96\% &  & 12.32 & 231867.49 & 0.04\%  &  & 31.81 & 40306.08  & 72.85\% \\  
 & 4 &  & 58.07 & 40601.52  & 8.97\%  &  & 64.43 & 231628.54 & -0.06\% &  & 473.26 & 29854.98  & 28.03\% \\  
 & 8 &  & 422.79 & 38515.49  & 3.37\%  &  & 247.52 & 231628.54 & -0.06\% &  & 3205 & 26178.52  & 12.27\% \\  
 & 16 &  & 1891.91  & 37602.58  & 0.92\%  &  & 418.06 & 231628.54 & -0.06\% &  & 10452.26 & 24357.1 & 4.45\% \\  
 & 32 &  & 5301.44  & 37267.23  & 0.02\%  &  & 891.85 & 231764.99 & 0.00\%  &  & 26822.58 & 23508.5 & 0.81\% \\  
 & 64 &  & 13494.45 & 37259.01  & - &  & 1050.57  & 231764.99 & - &  & 65099.79 & 23318.46  & 0.00\% \\  
 & dynamic   &  & 962.55 & 37716.6 & 1.23\%  &  & 40.95 & 231628.54 & -0.06\% &  & 4672.58  & 25418.94  & 9.01\% \\ \hline
\end{tabular}
\end{adjustbox}
\end{center}
\caption{Numerical results for the different RH look-ahead policies (with no discount factor).} 
\label{tab:RH_numerical_results}
\end{footnotesize}
\end{table}

\subsubsection{The plateauing effect and exponential growth in computational time. }
\label{subsubsec:plateauing}
The first observation is that in static RH policies, across different instances, the long-run average cost $\bar z_\tau$ \emph{plateaus} after a certain length of forecast horizon $\tau^*$. In other words, the decrease in $\bar z_\tau$ as $\tau$ increases becomes hardly noticeable after $\tau > \tau^*$. This behavior can be observed in Figure~\ref{fig:longrun_avg_cost_1H_static_without_offline} to Figure~\ref{fig:longrun_avg_cost_6H_static_without_offline} in the appendix. Across all of the instances, this plateauing behavior seems to take place somewhere between $8 \leq \tau^* \leq 16$ where $\bar z_\tau$ is around $2.21\%$, $2.52\%$, and $4.33\%$ larger than $\bar z_{64}$ on average, for instances where $|H|=1$, $|H|=3$ and $|H|=6$, respectively. On the other hand, we can see that static RH policies with $\tau \in [8, 16]$ require, on average, a computational time of 770.13 seconds, 1439.22 seconds, and 8016.52 seconds for instances where $|H|=1$, $|H|=3$ and $|H| = 6$, respectively. Comparing these to the static RH policy with $\tau = 64$, we can see that on average, it requires a computational time of 10569.23 seconds, 12368.98 seconds, and 42374.78 seconds for instances where $|H|=1$, $|H|=3$ and $|H|=6$, respectively. This means that for a $2.21\%$, $2.52\%$ and $4.33\%$ reduction in the optimality gap, using a static RH policy with $\tau = 64$ instead of $\tau \in [8, 16]$ would require a $92.71\%$, $88.36\%$ and $83.19\%$ increase in the computational time, on average, for instances where $|H|=1$, $|H|=3$ and $|H|=6$, respectively.

\subsubsection{Static vs. dynamic RH policies.}
\label{subsubsec:static_vs_dynamic}
Overall, across all of the instances, the dynamic RH policy performs quite well compared to the benchmark static RH policy with $\tau =64$. At its \emph{worst}, the $\bar z_{\text{dynamic}}$ is $9.57\%$ larger than $\bar z_{64}$, which happens in the instance where ($|H|=3, d_t = 1750$, $|\Xi_t|=12$), followed by the instance where ($|H|=6, d_t = 2250$, $|\Xi_t|=12$) and ($|H|=3, d_t = 1750$, $|\Xi_t|=5$), where the corresponding $\bar z_{\text{dynamic}}$ is $9.01\%$ and $7.02\%$ larger than $\bar z_{64}$, respectively. It is interesting to see that these three instances also happen to be the ones with the least accurate regression fit according to the (aggregated) coefficient of determination as shown in Table~\ref{tab:decomposed_fit}. Nevertheless, apart from these three instances, on average, $\bar z_{\text{dynamic}}$ is only around $1.37\%$ larger than $\bar z_{64}$ in the instances where $|H|=1$, $0.19\%$ larger in the remaining instances where $|H|=3$, and $4.97\%$ larger in the remaining instances where $|H|=6$. Comparing these gaps to the computational time required to compute the static RH policy with $\tau = 64$, we can see from Table~\ref{tab:RH_numerical_results} that for a $1.37\%$, $4.42\%$, and $5.98\%$ reduction in the optimality gap, it would require $91.48\%$, $97.06\%$, and $92.82\%$ increase in the computational time, for the instances where $|H|=1$, $|H|=3$ and $|H|=6$, respectively.
  
Moreover, given the monotonicity of $\bar z_\tau$ as $\tau$ increases, it might be instructive to explore where the dynamic policy ranks (in terms of its performance) compared to static RH policies with various lengths of forecast horizons. In nine out of the twelve instances that we have, the long-run average cost and the computational time by the dynamic policy are both placed between the static RH policies with $\tau = 8$ and $\tau = 16$. Even in instances where this is not the case -- namely, the instances ($|H|=3, d_t = 1750$, $|\Xi_t|=12$), ($|H|=3, d_t = 2250$, $|\Xi_t|=5$) and ($|H|=6, d_t = 2000$, $|\Xi_t|=5$) -- the performance of the dynamic RH policy, in terms of $\bar z_{\text{dynamic}}$ and $\text{Time}_{\text{dynamic}}$, places very closely to policies with $\tau =8$ in the case of the two instances with $|H|=3$ and very closely to policies with $\tau =16$ in the case of the instance with $|H|=6$. This time interval is, as discussed earlier, the same interval at which the plateauing behavior starts to take place. This coincidence, however, is somewhat expected given the fact that our regression functions are constructed precisely so that the corresponding states are mapped to the \emph{smallest} forecast horizon $\tau^*$ for which $\|a^\tau - a^{\tau^*}\| < \epsilon$ and hence $\|\bar z_{\tau} - \bar z_{\tau^*}\| < \epsilon$, i.e. when $\bar z_{\tau}$ plateaus. This provides a numerical evidence on the validity of our proposed approach.  
  
Finally, another advantage to this dynamic policy over the static policy is its \emph{usability}. To elaborate on this, it is important to note that when using a static policy, it is not clear how one can identify a sufficient forecast horizon $\tau^*$ a priori, for which $\bar z_{\tau}$ plateaus for all $\tau \geq \tau^*$. In our implementation, for instance, we do this by keeping track of $\bar z_\tau$ and testing the static RH policy with every $\tau \in \{1, 2, \dots, 64\}$. In other words, if the DM chooses to use a static policy, then prior to solving the problem he/she does not know what $\tau^*$ is, and the only way to find out is to solve the \emph{actual} problem for multiple values of $\tau$. While the offline training step involved in the estimation of the dynamic policy regression parameter, in some sense, also follows this same enumeration process, it is less expensive from a computational perspective and allows for more flexibility. First, when enumerating the values $\tau \in \mathbb{Z}_+$, we do not have to solve the \emph{actual} problem for every $t=1, 2, \dots, T$. Instead, we \emph{only} need to do this for the first stage problem under different sampled initial states $s^n$. Moreover, the accuracy of the regression fit can be adapted according to the sample size $N$ which depends on the available computational budget. Finally, the DM has the opportunity to study and analyze function $\Tau(s_t)$ prior to solving the \emph{actual} problem, e.g., he/she has the freedom to construct the regression functions that fit the set of sample initial states $s^{n}$ and the corresponding responses the best.
  
\subsubsection{RH policies vs. stationary policy.} 
\label{subsubsec:RH_vs_periodic}
From Table~\ref{tab:discount_numerical_results_1H} to Table~\ref{tab:discount_numerical_results_6H}, we can see that stationary policy trained with a given discount factor using the periodic variant of the SDDP algorithm performs quite well compared to the RH policies. Specifically, if we pick the policies corresponding to $\gamma = 0.99$ values in every instance and compare them to $\bar z_{64}$, we can see that $\bar z_\gamma$ is \emph{only} around $1.01\%, 1.41\%$, and $15.71\%$ larger than $\bar z_{64}$ on average, for instances where $|H|=1$, $|H|=3$ and $|H| = 6$, respectively. More importantly, the computational time required in the stationary policy is $66.50\%$, $73.15\%$, and $91.74\%$ less, on average, than that of the static RH policies with $\tau=64$, for instances where $|H|=1$, $|H|=3$ and $|H|=6$, respectively.
\begin{table}[htbp]
\begin{footnotesize}
\begin{center}
\begin{tabular}{|c|c|c|lcc|c|}
\hline
\textbf{$d_t$}  & \textbf{$|\Xi_t|$}  & \textbf{$\gamma$} & \textbf{Time} & \textbf{$\bar z_\gamma$} & \textbf{$(\bar z_\gamma - \bar z_{0.95})/\bar z_{0.95}\%$} & \multicolumn{1}{l|}{\textbf{$(\bar z_\gamma - \bar z_{64})/\bar z_{64}\%$}} \\ \hline
\multirow{22}{*}{1000} & \multirow{11}{*}{5}  & 0.10 & 0.19    & 12684.41 & 21.90\% & 33.22\%\\ 
                       &                      & 0.20 & 0.37    & 12684.41 & 21.90\% & 33.22\%\\ 
                       &                      & 0.30 & 0.53    & 12684.41 & 21.90\% & 33.22\%\\ 
                       &                      & 0.40 & 0.64    & 12684.41 & 21.90\% & 33.22\%\\ 
                       &                      & 0.50 & 0.80    & 12684.41 & 21.90\% & 33.22\%\\ 
                       &                      & 0.60 & 1.27    & 12684.41 & 21.90\% & 33.22\%\\ 
                       &                      & 0.70 & 1.62    & 10341.58 & 4.20\%  & 8.62\%\\
                       &                      & 0.80 & 2.91    & 10051.84 & 1.44\%  & 5.57\%\\
                       &                      & 0.90 & 14.94   & 9521.41  & -4.05\% & -    \\ 
                       &                      & 0.95 & 54.17   & 9513.41  & -4.14\% & -    \\ 
                       &                      & 0.99 & 1934.00 & 9906.97  & -       & 4.05\%  \\ \cline{2-7} 
                       &\multirow{11}{*}{12}  & 0.10 & 0.21    & 11346.83 & 27.98\% & 39.06\% \\
                       &                      & 0.20 & 0.41    & 11346.83 & 27.98\% & 39.06\%\\ 
                       &                      & 0.30 & 0.58    & 11346.83 & 27.98\% & 39.06\%\\ 
                       &                      & 0.40 & 0.77    & 11346.83 & 27.98\% & 39.06\%\\ 
                       &                      & 0.50 & 1.03    & 11004.30 & 25.74\% & 34.86\%\\ 
                       &                      & 0.60 & 1.72    & 10675.27 & 23.45\% & 30.83\%\\ 
                       &                      & 0.70 & 2.27    & 8845.99  & 7.62\%  & 8.41\%\\
                       &                      & 0.80 & 3.63    & 8571.37  & 4.66\%  & 5.05\%\\
                       &                      & 0.90 & 22.39   & 8216.27  & 0.54\%  & 0.69\%\\
                       &                      & 0.95 & 97.13   & 8188.86  & -       & -    \\ 
                       &                      & 0.99 & 5391.60 & 8171.77  & -       & -       \\ \hline
\multirow{22}{*}{1500} & \multirow{11}{*}{5}  & 0.10 & 0.23    & 65651.50 & 38.17\% & 62.19\%\\ 
                       &                      & 0.20 & 0.46    & 58755.27 & 30.91\% & 45.15\%\\ 
                       &                      & 0.30 & 0.54    & 58647.44 & 30.78\% & 44.89\%\\ 
                       &                      & 0.40 & 0.70    & 54744.80 & 25.85\% & 35.25\%\\ 
                       &                      & 0.50 & 0.69    & 46043.49 & 11.83\% & 13.75\%\\ 
                       &                      & 0.60 & 1.02    & 43352.34 & 6.36\%  & 7.10\%\\
                       &                      & 0.70 & 1.48    & 42521.26 & 4.53\%  & 5.05\%\\
                       &                      & 0.80 & 2.39    & 41663.21 & 2.56\%  & 2.93\%\\
                       &                      & 0.90 & 12.14   & 40584.36 & -       & -    \\ 
                       &                      & 0.95 & 47.02   & 40617.51 & -       & -    \\ 
                       &                      & 0.99 & 2185.29 & 40594.81 & -       & -       \\ \cline{2-7} 
                       & \multirow{11}{*}{12} & 0.10 & 0.26    & 58461.04 & 36.37\% & 56.90\%\\ 
                       &                      & 0.20 & 0.52    & 58461.04 & 36.37\% & 56.90\%\\ 
                       &                      & 0.30 & 0.59    & 51964.42 & 28.42\% & 39.47\%\\ 
                       &                      & 0.40 & 0.86    & 47407.99 & 21.54\% & 27.24\%\\ 
                       &                      & 0.50 & 0.95    & 45398.62 & 18.07\% & 21.85\%\\ 
                       &                      & 0.60 & 1.57    & 40330.77 & 7.77\%  & 8.24\%\\
                       &                      & 0.70 & 2.03    & 43005.42 & 13.51\% & 15.42\%\\ 
                       &                      & 0.80 & 2.95    & 37759.45 & 1.49\%  & 1.34\%\\
                       &                      & 0.90 & 16.13   & 37305.85 & -       & -    \\ 
                       &                      & 0.95 & 72.57   & 37269.70 & -       & -    \\ 
                       &                      & 0.99 & 4651.74 & 37196.90 & -       & -       \\ \hline
\end{tabular}
\end{center}
\caption{Numerical results for the stationary policy (with discount factor $\gamma$) for $|H|=1$.} 
\label{tab:discount_numerical_results_1H}
\end{footnotesize}
\end{table}

\begin{table}[htbp]
\begin{footnotesize}
\begin{center}
\begin{tabular}{|c|c|c|lcc|c|}
\hline
\textbf{$d_t$}  & \textbf{$|\Xi_t|$}  & \textbf{$\gamma$} & \textbf{Time} & \textbf{$\bar z_\gamma$} & \textbf{$(\bar z_\gamma - \bar z_{0.95})/\bar z_{0.95}\%$} & \multicolumn{1}{l|}{\textbf{$(\bar z_\gamma - \bar z_{64})/\bar z_{64}\%$}} \\ \hline
\multirow{22}{*}{1750} & \multirow{11}{*}{5}  & 0.10 & 0.25    & 72329.48  & 48.80\% & 105.02\%\\ 
                       &                      & 0.20 & 0.50    & 61305.31  & 39.60\% & 73.77\%\\
                       &                      & 0.30 & 0.56    & 58958.97  & 37.19\% & 67.12\%\\
                       &                      & 0.40 & 0.78    & 45323.61  & 18.30\% & 28.47\%\\
                       &                      & 0.50 & 0.75    & 42397.72  & 12.66\% & 20.18\%\\
                       &                      & 0.60 & 1.15    & 37542.00  & 1.36\%  & 6.41\%\\ 
                       &                      & 0.70 & 1.70    & 36528.44  & -1.37\% & 3.54\%\\ 
                       &                      & 0.80 & 2.55    & 35771.06  & -3.52\% & 1.39\%\\ 
                       &                      & 0.90 & 13.38   & 36459.97  & -1.56\% & 3.34\%\\ 
                       &                      & 0.95 & 43.60   & 37036.00  & -       & 4.98\%\\ 
                       &                      & 0.99 & 1938.69 & 37029.61  & -       & 4.96\%   \\ \cline{2-7} 
                       & \multirow{11}{*}{12} & 0.10 & 0.27    & 87359.15  & 48.74\% & 96.39\% \\
                       &                      & 0.20 & 0.54    & 87359.15  & 48.74\% & 96.39\%\\
                       &                      & 0.30 & 0.67    & 75697.62  & 40.84\% & 70.18\%\\
                       &                      & 0.40 & 0.91    & 75597.00  & 40.76\% & 69.95\%\\
                       &                      & 0.50 & 1.03    & 55065.02  & 18.67\% & 23.79\%\\
                       &                      & 0.60 & 1.65    & 47649.75  & 6.01\%  & 7.12\%\\ 
                       &                      & 0.70 & 2.67    & 47572.81  & 5.86\%  & 6.95\%\\ 
                       &                      & 0.80 & 4.36    & 45281.33  & 1.10\%  & 1.80\%\\ 
                       &                      & 0.90 & 19.38   & 44618.47  & -       & -     \\ 
                       &                      & 0.95 & 77.11   & 44853.49  & -       & 0.84\%\\ 
                       &                      & 0.99 & 5021.77 & 44783.72  & -       & 0.68\%   \\ \hline
\multirow{22}{*}{2250} & \multirow{11}{*}{5}  & 0.10 & 0.23    & 222307.24 & 16.62\% & 19.80\%\\
                       &                      & 0.20 & 0.47    & 222307.24 & 16.62\% & 19.80\%\\
                       &                      & 0.30 & 0.59    & 209145.23 & 11.38\% & 12.71\%\\
                       &                      & 0.40 & 0.79    & 208001.60 & 10.89\% & 12.09\%\\
                       &                      & 0.50 & 0.86    & 190872.27 & 2.89\%  & 2.86\%\\ 
                       &                      & 0.60 & 1.41    & 187446.56 & 1.12\%  & 1.02\%\\ 
                       &                      & 0.70 & 2.39    & 186434.33 & 0.58\%  & -     \\ 
                       &                      & 0.80 & 3.58    & 185552.76 & -       & -     \\ 
                       &                      & 0.90 & 15.25   & 185434.50 & -       & -     \\ 
                       &                      & 0.95 & 55.98   & 185467.83 & -       & -     \\ 
                       &                      & 0.99 & 1985.85 & 185351.59 & -       & - \\ \cline{2-7} 
                       & \multirow{11}{*}{12} & 0.10 & 0.26    & 246633.48 & 6.09\%  & 6.42\%\\ 
                       &                      & 0.20 & 0.53    & 246633.48 & 6.09\%  & 6.42\%\\ 
                       &                      & 0.30 & 0.67    & 244842.70 & 5.41\%  & 5.64\%\\ 
                       &                      & 0.40 & 0.84    & 244842.70 & 5.41\%  & 5.64\%\\ 
                       &                      & 0.50 & 1.06    & 231806.51 & -       & -     \\ 
                       &                      & 0.60 & 1.88    & 231710.85 & -       & -     \\ 
                       &                      & 0.70 & 3.25    & 231691.90 & -       & -     \\ 
                       &                      & 0.80 & 5.45    & 231691.90 & -       & -     \\ 
                       &                      & 0.90 & 26.18   & 231690.23 & -       & -     \\ 
                       &                      & 0.95 & 100.28  & 231691.90 & -       & -     \\ 
                       &                      & 0.99 & 4337.70 & 231602.68 & -       & - \\ \hline

\end{tabular}
\end{center}
\caption{Numerical results for the stationary policy (with discount factor $\gamma$) for $|H|=3$.} 
\label{tab:discount_numerical_results_3H}
\end{footnotesize}
\end{table}

\begin{table}[htbp]
\begin{footnotesize}
\begin{center}
\begin{tabular}{|c|c|c|lcc|c|}
\hline
\textbf{$d_t$}  & \textbf{$|\Xi_t|$}  & \textbf{$\gamma$} & \textbf{Time} & \textbf{$\bar z_\gamma$} & \textbf{$(\bar z_\gamma - \bar z_{0.99})/\bar z_{0.99}\%$} & \multicolumn{1}{l|}{\textbf{$(\bar z_\gamma - \bar z_{64})/\bar z_{64}\%$}} \\ \hline
\multirow{22}{*}{2000} & \multirow{11}{*}{5}  & 0.10 & 0.34    & 43322.32 & 67.14\%  & 236.97\%\\ 
                       &                      & 0.20 & 0.68    & 31772.28 & 55.19\%  & 147.13\%\\ 
                       &                      & 0.30 & 0.77    & 24250.31 & 41.30\%  & 88.62\%\\
                       &                      & 0.40 & 1.01    & 18149.80 & 21.56\%  & 41.17\%\\
                       &                      & 0.50 & 0.98    & 16971.14 & 16.12\%  & 32.00\%\\
                       &                      & 0.60 & 1.51    & 14068.05 & -1.19\%  & 9.42\%\\ 
                       &                      & 0.70 & 1.97    & 13288.72 & -7.13\%  & 3.36\%\\ 
                       &                      & 0.80 & 3.25    & 12938.78 & -10.02\% & 0.64\%\\ 
                       &                      & 0.90 & 16.13   & 13218.13 & -7.70\%  & 2.81\%\\ 
                       &                      & 0.95 & 60.61   & 14188.20 & - & 10.36\%\\
                       &                      & 0.99 & 2405.77 & 14235.84 & - & 10.73\%\\ \cline{2-7} 
                       & \multirow{11}{*}{12} & 0.10 & 0.37    & 32429.25 & 62.25\%  & 170.63\% \\ 
                       &                      & 0.20 & 0.74    & 29770.94 & 58.88\%  & 148.45\%\\ 
                       &                      & 0.30 & 0.92    & 16466.87 & 25.65\%  & 37.42\%\\
                       &                      & 0.40 & 1.12    & 14525.55 & 15.72\%  & 21.22\%\\
                       &                      & 0.50 & 1.40    & 15746.90 & 22.25\%  & 31.41\%\\
                       &                      & 0.60 & 2.15    & 13932.88 & 12.13\%  & 16.27\%\\
                       &                      & 0.70 & 3.25    & 12764.11 & 4.09\%   & 6.52\%\\ 
                       &                      & 0.80 & 4.96    & 12134.38 & -0.89\%  & 1.26\%\\ 
                       &                      & 0.90 & 24.48   & 11997.69 & -2.04\%  & -     \\ 
                       &                      & 0.95 & 97.90   & 12175.00 & -0.55\%  & 1.60\%\\ 
                       &                      & 0.99 & 5250.73 & 12242.54 & - & 2.17\%   \\ \hline
\multirow{22}{*}{2250} & \multirow{11}{*}{5}  & 0.10 & 0.40    & 77787.89 & 56.53\%  & 212.57\%\\ 
                       &                      & 0.20 & 0.80    & 52850.73 & 36.02\%  & 112.37\%\\ 
                       &                      & 0.30 & 0.79    & 48205.23 & 29.86\%  & 93.70\%\\
                       &                      & 0.40 & 0.96    & 34184.51 & 1.09\%   & 37.36\%\\
                       &                      & 0.50 & 1.09    & 32246.83 & -4.86\%  & 29.58\%\\
                       &                      & 0.60 & 1.50    & 27511.85 & -22.90\% & 10.55\%\\
                       &                      & 0.70 & 2.17    & 25430.95 & -32.96\% & 2.19\%\\ 
                       &                      & 0.80 & 3.27    & 24990.33 & -35.30\% & -     \\ 
                       &                      & 0.90 & 15.90   & 26020.24 & -29.95\% & 4.56\%\\ 
                       &                      & 0.95 & 56.71   & 29517.01 & -14.55\% & 18.61\%\\
                       &                      & 0.99 & 1558.71 & 33812.70 & - & 35.87\%  \\ \cline{2-7} 
                       & \multirow{11}{*}{12} & 0.10 & 0.40    & 73726.60 & 63.92\%  & 216.17\%\\ 
                       &                      & 0.20 & 0.79    & 45434.47 & 41.45\%  & 94.84\%\\
                       &                      & 0.30 & 0.91    & 41025.18 & 35.16\%  & 75.93\%\\
                       &                      & 0.40 & 1.14    & 36716.93 & 27.55\%  & 57.46\%\\
                       &                      & 0.50 & 1.34    & 30739.56 & 13.47\%  & 31.82\%\\
                       &                      & 0.60 & 1.97    & 25950.83 & -2.50\%  & 11.29\%\\
                       &                      & 0.70 & 3.15    & 24211.81 & -9.86\%  & 3.83\%\\ 
                       &                      & 0.80 & 5.02    & 23548.44 & -12.96\% & 0.99\%\\ 
                       &                      & 0.90 & 25.96   & 23562.13 & -12.89\% & 1.04\%\\ 
                       &                      & 0.95 & 99.43   & 24407.57 & -8.98\%  & 4.67\%\\ 
                       &                      & 0.99 & 4784.02 & 26599.80 & - & 14.07\%  \\ \hline
\end{tabular}
\end{center}
\caption{Numerical results for the stationary policy (with discount factor $\gamma$) for $|H|=6$.} 
\label{tab:discount_numerical_results_6H}
\end{footnotesize}
\end{table}
It is also interesting to see, as shown in Figure~\ref{fig:longrun_avg_costH_discount} to Figure~\ref{fig:longrun_avg_cost_6H_discount} in the appendix, that the performance of the stationary policy exhibits a plateauing behavior (as $\gamma$ increases) similar to that of the static RH policies (as $\tau$ increases). Although not with the same monotonicity rate, it is important to note that when implementing the periodic variant of the SDDP algorithm under the \emph{stationarity} assumption (Assumption~\ref{assm:stationarity}) with $m=1$, the (infinite horizon) MSP reduces to a static problem. This leads to the delicate issue of how to choose the trial points during the forward pass of the SDDP algorithm, which should depend on the length of the horizon $T$. Approximating $T=\infty$ using a finite horizon leads to an error which is a consequence of the so-called end-of-horizon effect. Moreover, when $\gamma$ is close to 1, the convergence of the algorithm becomes slow since the needed $T$ can be large. To deal with this, in our implementation of the periodic variant of the SDDP algorithm, we treat $T$ as a random variable following a \emph{geometric} distribution with a success probability $p = (1-\gamma)$. Then, at the beginning of every iteration $i$, we sample a forecast horizon $T^i$, a sample path $\xi^i = (\xi^i_1, \dots, \xi^i_{T^i})$ and implement the forward/backward pass along $\xi^i$. 
Note that the sampling error coming from sampling $T^i$ is what might be causing the non-monotonicity of $\bar z_\gamma$ as $\gamma$ increases in our numerical results.

Finally, we investigate how the $\epsilon-$sufficient forecast horizon $\tau^*_{\epsilon}$ suggested by Theorem~\ref{thm:bound} (see also~\eqref{eq:sufficient_tau}) compared to the average length of forecast horizon used by the dynamic RH policy $\bar \tau = \sum_{t=1}^T\tau_t/T$ over the entire sample path in the out-of-sample test, as shown in Table~\ref{tab:dynamic_vs_theory}. Note that the upper bound $\kappa$ on the immediate cost function $f_t(\cdot,\cdot)$ is calculated as the optimal objective value of problem~\eqref{eq:HPOP} with state variables $s_{h,t} = 0, \; \forall h \in H$. From Table~\ref{tab:dynamic_vs_theory}, we see that $\bar \tau$ has an average value of around $10.1, 5.5$, and $11.47$ for instances when $|H| = 1, |H| = 3$, and $|H| = 6$, respectively. These numbers are not only significantly smaller than those of the $\tau^*_{\epsilon}$ corresponding to the different discount factors computed via Theorem~\ref{thm:bound} and equation~\eqref{eq:sufficient_tau}, but also consistent with the observations made in Subsection~\ref{subsubsec:static_vs_dynamic} regarding where the dynamic policy ranks (in terms of its performance) compared to static RH policies with various lengths of forecast horizons. 

\begin{table}[htbp]
\begin{footnotesize}
\begin{center}
\begin{adjustbox}{width=\textwidth}
\begin{tabular}{|c|c|c|c|ccccccccccc|}
\hline 
\multirow{3}{*}{$|H|$} & \multirow{3}{*}{$\tilde d_t$} & \multirow{3}{*}{$\bar \tau$} & \multirow{3}{*}{$\kappa$} & \multicolumn{11}{c|}{$\gamma$}                                                              \\ \cline{5-15} 
                       &                               &                              &                           & 0.10  & 0.20  & 0.30  & 0.40  & 0.50  & 0.60  & 0.70  & 0.80   & 0.90   & 0.95   & 0.99    \\ \cline{5-15} 
                       &                               &                              &                           & \multicolumn{11}{c|}{$\tau_{\gamma}$}                                                       \\ \hline
\multirow{2}{*}{1}     & 1000                          & 8.93                         & 53000                     & 9.77  & 14.05 & 18.89 & 24.99 & 33.30 & 45.63 & 66.15 & 107.56 & 234.37 & 494.93 & 2686.09 \\
                       & 1500                          & 11.27                        & 260500                    & 10.46 & 15.04 & 20.22 & 26.73 & 35.60 & 48.74 & 70.62 & 114.69 & 249.49 & 525.98 & 2844.53 \\ \hline
\multirow{2}{*}{3}     & 1750                          & 8.42                         & 385500                    & 10.63 & 15.28 & 20.54 & 27.16 & 36.17 & 49.51 & 71.72 & 116.45 & 253.20 & 533.62 & 2883.52 \\
                       & 2250                          & 2.58                         & 635500                    & 10.85 & 15.59 & 20.96 & 27.71 & 36.89 & 50.49 & 73.12 & 118.69 & 257.95 & 543.36 & 2933.26 \\ \hline
\multirow{2}{*}{6}     & 2000                          & 12.26                        & 412000                    & 10.66 & 15.33 & 20.60 & 27.23 & 36.26 & 49.64 & 71.90 & 116.75 & 253.84 & 534.91 & 2890.14 \\
                       & 2250                          & 10.68                        & 537000                    & 10.78 & 15.49 & 20.82 & 27.52 & 36.64 & 50.16 & 72.65 & 117.93 & 256.35 & 540.08 & 2916.50 \\ \hline 
\end{tabular}
\end{adjustbox}
\end{center}
\caption{The average number of forecast horizon in the dynamic policy $\bar \tau = \sum_{t=1}^T\tau_t/T$ and the $\epsilon-$sufficient forecast horizon $\tau^*_{\epsilon}$ obtained by~\eqref{eq:sufficient_tau}.} 
\label{tab:dynamic_vs_theory}
\end{footnotesize}
\end{table}

\subsection{Sensitivity analysis}
\label{subsec:sensitivity_analysis}
In this section, we present a sensitivity analysis to assess the robustness of the solution with respect to different values that we use for the \textit{stalling} parameter $w$. As previously noted, the \textit{stalling} parameter $w$ is what we use to determine whether or not $x^{\tau}_1$ will remain the same for all $\tau \geq \tau^*$, during the offline training step of our proposed approach for the dynamic RH policy (see also \textbf{Step} 3.3 in Algorithm~\ref{alg:ADP}). In Table~\ref{tab:sensitivity_analysis}, in addition to the results obtained by using the stalling parameter $w=10$ used in our original implementation, we also performed the same procedure using $w \in \{5,15\}$. As it is intuitive to consider $w=15$ as the stalling parameter with the \textit{most stable} solution, we use it as the benchmark. Similarly to the analysis presented in Subsection~\ref{subsec:results}, we also report (i) the training \textbf{time}$_w$ (in \emph{seconds}) for $w=15$; (ii) the \textbf{long-run average cost}: $\bar z_w$ given by~\eqref{eq:long_run_avg} for $w=15$; and (iii) the \emph{relative} \textbf{gap} in \textbf{time}$_w$ and $\bar z_w$ for $w \in \{5,10\}$ compared to \textbf{time}$_{15}$ and $\bar z_{15}$.

\subsubsection{Sensitivity of the long-run average cost.}
\label{susubbsec:long_run_avg}
When comparing the performance in terms of $\bar z_{w},\; \forall w \in \{5,10\}$ and the $\bar z_{15}$ obtained by using the supposedly most stable parameter $w=15$, we can see from Table~\ref{tab:sensitivity_analysis} that the difference is negligible. Specifically, compared to the $\bar z_{w}$ obtained by the $W=5$, $\bar z_{15}$ has a relative gap of $1.73\%$, $3.60\%$, and $3.47\%$ when averaged across the instances where $|H|=1$, $|H|=3$ and $|H|=6$, respectively. Whereas, compared to the $\bar z_{w}$ obtained by the $W=10$ used in our original implementation, $\bar z_{15}$ has a relative gap of $-2.41\%$, $-3.91\%$, and $0.14\%$ when averaged across the instances where $|H|=1$, $|H|=3$ and $|H|=6$, respectively.

\exclude{Although the overall difference in the performances is not, as one might expect, the performance of the supposedly \textit{most} stable parameter $w=15$ is better, on average, in terms of $\bar z_w$ compared to the supposedly \textit{least} stable parameter $w=5$. Whereas, compared to $w=10$ the performance fluctuates between worse and better. While it might seem unexpected that $w=15$ yields worse performances compared to $w=5$ and $w=15$ we suspect that this can be attributed to the estimation errors. }

\subsubsection{Sensitivity of the online training time.}
\label{susubbsec:time}
Unlike the difference in the performance in terms of $\bar z_{w}$, the difference in the time it took each policy to perform the RH procedure seems to be significant. Specifically, compared to the Time$_{w}$ obtained for $w=5$, Time$_{15}$ has a relative gap of $-63.26\%$, $-38.70\%$, and $-61.88\%$ when averaged across the instances where $|H|=1$, $|H|=3$ and $|H|=6$, respectively. Whereas, compared to the Time$_{w}$ obtained by the $w=10$ used in our original implementation, Time$_{15}$ has a relative gap of $117.89\%$, $-35.49\%$, and $-12.05\%$ when averaged across the instances where $|H|=1$, $|H|=3$ and $|H|=6$, respectively.

As we can see, the dynamic RH policy constructed by using $w=5$ has much less computational time compared to the dynamic RH policy constructed by using $w=15$. This is to be expected, however, since using $w=5$ is likely to prescribe smaller forecast horizons during the RH procedure and hence lead to less computational time overall. Similarly, apart from the instances where $|H|=1$, we can see that the dynamic RH policy constructed by using $w=10$ has less computational time compared to the one with $w=15$. Overall, using $w=5$ instead of $w=10$ will lead to an average of $77.04\%, -23.50\%$ and $56.40\%$ decrease in the computational time for instances where $|H|=1$, $|H|=3$ and $|H|=6$, respectively. However, it will also lead to an average of $4.26\%, 8.68\%$, and $3.33\%$ increase in the optimality gap for these instances respectively.

\begin{table}[htbp]
\begin{footnotesize}
\begin{center}
\begin{tabular}{|c|c|c|cc|cc|cc|}
\hline
\multirow{2}{*}{$|H|$} & \multirow{2}{*}{$|\Xi_t|$}  & \multirow{2}{*}{$d_t$} & $\bar z_{w}$ & \textbf{Time$_{w}$} & \multicolumn{2}{c|}{$(\bar z_w - \bar z_{15})/\bar z_{15}\%$} & \multicolumn{2}{c|}{$(\text{Time}_w - \text{Time}_{15})/\text{Time}_{15}\%$} \\ \cline{4-9} 
   &   & & \multicolumn{2}{c|}{$w=15$} & $w=5$  & $w=10$   & $w=5$ & $w=10$ \\ \hline
\multirow{4}{*}{1} & \multirow{2}{*}{5}  & 1000  & 10238.65 & 114.27 & -2.24\%  & -6.46\%  & -47.32\%  & 192.28\% \\  
   &   & 1500  & 42024.78 & 544.37 & 5.42\% & -1.97\%  & -72.12\%  & -8.66\% \\
   & \multirow{2}{*}{12} & 1000  & 8404.53  & 432.41 & 1.37\% & -1.07\%  & -56.44\%  & 318.52\% \\  
   &   & 1500  & 37770.62 & 1386.75  & 2.38\% & -0.14\%  & -77.17\%  & -30.59\% \\ \hline
\multirow{4}{*}{3} & \multirow{2}{*}{5}  & 1750  & 43949.17 & 1010.73  & 14.41\%  & -14.09\% & -85.15\%  & -34.53\% \\ 
   &   & 2250  & 190165.57  & 42.50  & -0.44\%  & -1.99\%  & -11.90\%  & -9.74\% \\  
   & \multirow{2}{*}{12} & 1750  & 48421.16 & 891.34 & 0.46\% & 0.65\% & -17.39\%  & -19.71\% \\  
   &   & 2250  & 232124.68  & 185.94 & -0.04\%  & -0.21\%  & -40.35\%  & -77.98\% \\ \hline
\multirow{4}{*}{6} & \multirow{2}{*}{5}  & 2000  & 13526.56 & 2383.90  & 1.46\% & -2.42\%  & -51.27\%  & 8.43\% \\  
   &   & 2250  & 27025.64 & 2671.07  & 3.28\% & 0.24\% & -61.50\%  & -27.78\% \\ 
   & \multirow{2}{*}{12} & 2000  & 12263.77 & 5337.97  & 5.86\% & 1.01\% & -64.81\%  & -9.76\% \\  
   &   & 2250  & 24988.45 & 5774.15  & 3.26\% & 1.72\% & -69.92\%  & -19.08\% \\ \hline
\end{tabular}
\end{center}

\caption{Sensitivity analysis with respect to the \textit{stalling} parameter $w$ used in the offline training step for the dynamic RH policy.} 
\label{tab:sensitivity_analysis}
\end{footnotesize}
\end{table}

\section{Conclusion}
\label{sec:conclusion}
In this paper, we have studied the question of how many stages to include in the forecast horizon when solving MSP problems using the RH procedure for both finite and infinite horizon MSP problems. In the infinite horizon discounted case, given a fixed forecast horizon $\tau$, we have shown that the resulting optimality gap associated with the static look-ahead policy in terms of the total expected discounted cost can be bounded by a function of the chosen forecast horizon $\tau$. This function could be used to provide an upper bound on forecast horizon $\tau^*_{\epsilon}$ where the corresponding static RH policy achieves a prescribed $\epsilon$ optimality gap. In the finite horizon case (with no discount), we have taken an ADP perspective and developed a dynamic RH policy where the forecast horizon $\tau$ to use in each roll of the RH procedure is chosen dynamically according to the state of the system at that stage via certain regression function that can be trained offline. This dynamic RH policy has shown to be capable of exploiting the system state related information encountered during the RH procedure. Our numerical results have illustrated the empirical behaviors of the proposed RH policies on a class of MSPs and highlighted the effectiveness of the proposed approaches. Specifically, when using a static RH policy where the length of the forecast horizon in every roll is fixed, we have shown that the optimality gap plateaued and ceased to respond to the increase in the length of the forecast horizon beyond a certain number of stages. On the other hand, we have shown that advantage of the dynamic RH policy against alternative approaches in the slight increase in the optimality gap compared to the significant reduction in the computational time.

We have identified several avenues for future research. First, we anticipate that the proposed approaches can be extended to various strategies to address the ``end-of-horizon'' effect during the RH procedure. Second, we expect that some assumptions made for the sake of simplicity in experiments and analysis can be relaxed, allowing the proposed approaches to be applied to a broader class of MSPs, including those defined on hidden Markov chains. Lastly, although we chose not to pursue in this paper, we expect that the proposed approaches can be more appealing for multi-stage stochastic integer programs with a finite action/control space.

\bibliographystyle{unsrt}  
\bibliography{references}  %%% Remove comment to use the external .bib file (using bibtex).
\newpage
\appendix
\section{Tables}
\begin{table}[htbp]
\begin{footnotesize}
\begin{center}
\begin{tabular}{ p{2.5cm}p{12.5cm}}
\hline
  Notation & Description\\
 \hline
$H, F$ & The set of hydro/thermal plants in the HPOP problem \\
$h, f$ & A hydro/thermal plant index in the set of hydro/thermal plants \\
$\vec{c}_{g, t}, \vec{\tilde c}_{g, t}$ & Cost and random cost vector of generating thermal power at time $t$ \\
 $c_{p, t}, \tilde c_{p, t}$ & Unit penalty cost and random unit penalty cost for unsatisfied demand at time $t$\\
 $b_{h, t}, \tilde b_{h, t} $ & Amount of inflows and random inflows to hydro plant $h$ during stage $t$ \\
 $r_{h, t} $ & Amount of power generated by releasing one unit of water flow in hydro plant $h$ \\
 $d_t, d_t$ & Demand and random demand in stage $t$ \\
 $\bar q_{h}$ & Maximum allowed amount of turbined flow in hydro plant $h$ in stage $t$ \\ 
 $\underline{v}_{h, t}$, $\bar v_{h, t}$ & Minimum/Maximum level of water allowed in hydro plant $h$ \\
  $\underline{g}_{h, t}$, $\bar g_{h, t}$ & Minimum/Maximum allowed amount of power generated by thermal plant $f$ in stage $t$ \\
 $U(h)$, $L(h)$ & Set of immediate upper/lower stream hydro plants of $h$ in the network \\
 \hline
\end{tabular}
\end{center}
\caption{Notation for the parameters of the HPOP stage-$t$ problem~\eqref{eq:HPOP}.}
\label{tab:HPOP_notation}
\end{footnotesize}
\end{table}

\begin{table}[htbp]
\begin{footnotesize}
\begin{center}
\begin{tabular}{c|cccccc|c}
\hline
  & $h_1$  & $h_2$  & $h_3$ & $h_4$  & $h_5$ & $h_6$ & $\mathbb{P}(\xi^k_t \in \Xi_t)$ \\ \hline
$\xi^1_t$ & 245.50 & 125.20 & 1438.00 & 311.00 & 16.20 & 29.70 & 0.20 \\
$\xi^2_t$ & 201.70 & 103.90 & 1085.30 & 221.90 & 13.00 & 23.60 & 0.15 \\
$\xi^3_t$ & 158.00 & 82.60  & 732.50  & 132.70 & 9.90  & 17.50 & 0.30 \\
$\xi^4_t$ & 130.20 & 58.60  & 488.10  & 93.10  & 7.00  & 10.70 & 0.15 \\
$\xi^5_t$ & 102.40 & 34.60  & 243.60  & 53.40  & 4.20  & 3.90  & 0.20 \\ \hline
\end{tabular}
\end{center}
\caption{Realizations \& probability distribution of $\xi_t$ (rainfall in hydro $h \in H$) when $|\Xi_t|=5$.}
\label{tab:5_real_dist}
\end{footnotesize}
\end{table}

\begin{table}[htbp]
\begin{footnotesize}
\begin{center}
\begin{tabular}{c|cccccc|c}
\hline
   & $h_1$  & $h_2$  & $h_3$ & $h_4$  & $h_5$ & $h_6$ & $\mathbb{P}(\xi^k_t \in \Xi_t)$ \\ \hline
$\xi^1_t$  & 245.50 & 125.20 & 1438.00 & 120.00 & 16.20 & 29.70 & 0.09 \\
$\xi^2_t$  & 232.50 & 117.00 & 1329.50 & 111.00 & 15.10 & 27.40 & 0.10 \\
$\xi^3_t$  & 219.40 & 108.70 & 1220.90 & 101.90 & 14.00 & 25.00 & 0.10 \\
$\xi^4_t$  & 206.40 & 100.50 & 1112.30 & 92.90  & 12.90 & 22.70 & 0.09 \\
$\xi^5_t$  & 193.40 & 92.30  & 1003.70 & 83.90  & 11.80 & 20.30 & 0.07 \\
$\xi^6_t$  & 180.40 & 84.00  & 895.10  & 74.80  & 10.70 & 18.00 & 0.06 \\
$\xi^7_t$  & 167.40 & 75.80  & 786.60  & 65.80  & 9.70  & 15.60 & 0.06 \\
$\xi^8_t$  & 154.40 & 67.50  & 678.00  & 56.70  & 8.60  & 13.30 & 0.07 \\
$\xi^9_t$  & 141.40 & 59.30  & 569.40  & 47.70  & 7.50  & 10.90 & 0.09 \\
$\xi^{10}_t$ & 128.40 & 51.10  & 460.80  & 38.70  & 6.40  & 8.60  & 0.10 \\
$\xi^{11}_t$ & 115.40 & 42.80  & 352.20  & 29.60  & 5.30  & 6.20  & 0.10 \\
$\xi^{12}_t$ & 102.40 & 34.60  & 243.60  & 20.60  & 4.20  & 3.90  & 0.09 \\ \hline
\end{tabular}
\end{center}
\caption{Realizations \& probability distribution of $\xi_t$ (rainfall in hydro $h \in H$) when $|\Xi_t|=12$.}
\label{tab:12_real_dist}
\end{footnotesize}
\end{table}

\begin{table}[htbp]
\begin{footnotesize}
\begin{center}
\begin{tabular}{c|cccccc}
\hline
 & $h_1$ & $h_2$ & $h_3$ & $h_4$ & $h_5$ & $h_6$ \\ \hline
$r_{h,t}$  & 0.18  & 0.35  & 0.75  & 0.32  & 0.56  & 0.15 \\
$\bar q_{h,t}$   & 220.00  & 585.00  & 1688.00 & 5220.00 & 2028.00 & 1480.00 \\
$\bar v_{h, t}$  & 672.00  & - & 17217.00  & 2500.00 & - & - \\
$\underline{v}_{h, t}$ & 0.00  & - & 0.00  & 0.00  & - & - \\
$x_0$  & 336.00  & - & 10330.20  & 1250.00 & - & - \\ \hline
$U(h)$   & $\emptyset$ & $\{1\}$ & $\emptyset$ & $\{2,3\}$ & $\emptyset$ & $\{4,5\}$ \\
$L(h)$   & $\{2\}$ & $\{4\}$ & $\{4\}$ & $\{6\}$ & $\{6\}$ & $\emptyset$ \\ \hline
\end{tabular}
\end{center}
\caption{A summary of the hydro plants data.}
\label{tab:hydro_data}
\end{footnotesize}
\end{table}

\begin{table}[htbp]
\begin{footnotesize}
\begin{center}
\begin{tabular}{c|cccc}
\hline
 & $f_1$ & $f_2$  & $f_3$ & $f_4$ \\ \hline
$\bar g_{f,t}$ & 20.00& 20.00  & 20.00 &  20.00 \\
$ c_{g,t}$& 20.00& 40.00  & 80.00 &  160.00 \\ \hline

\end{tabular}
\end{center}
\caption{A summary of the thermal plants data.}\label{tab:thermo_data}
\end{footnotesize}
\end{table}

\section{Figures}
\label{append:figures}

\begin{figure}[htbp]
\begin{center}
\includegraphics[width=0.6\textwidth]{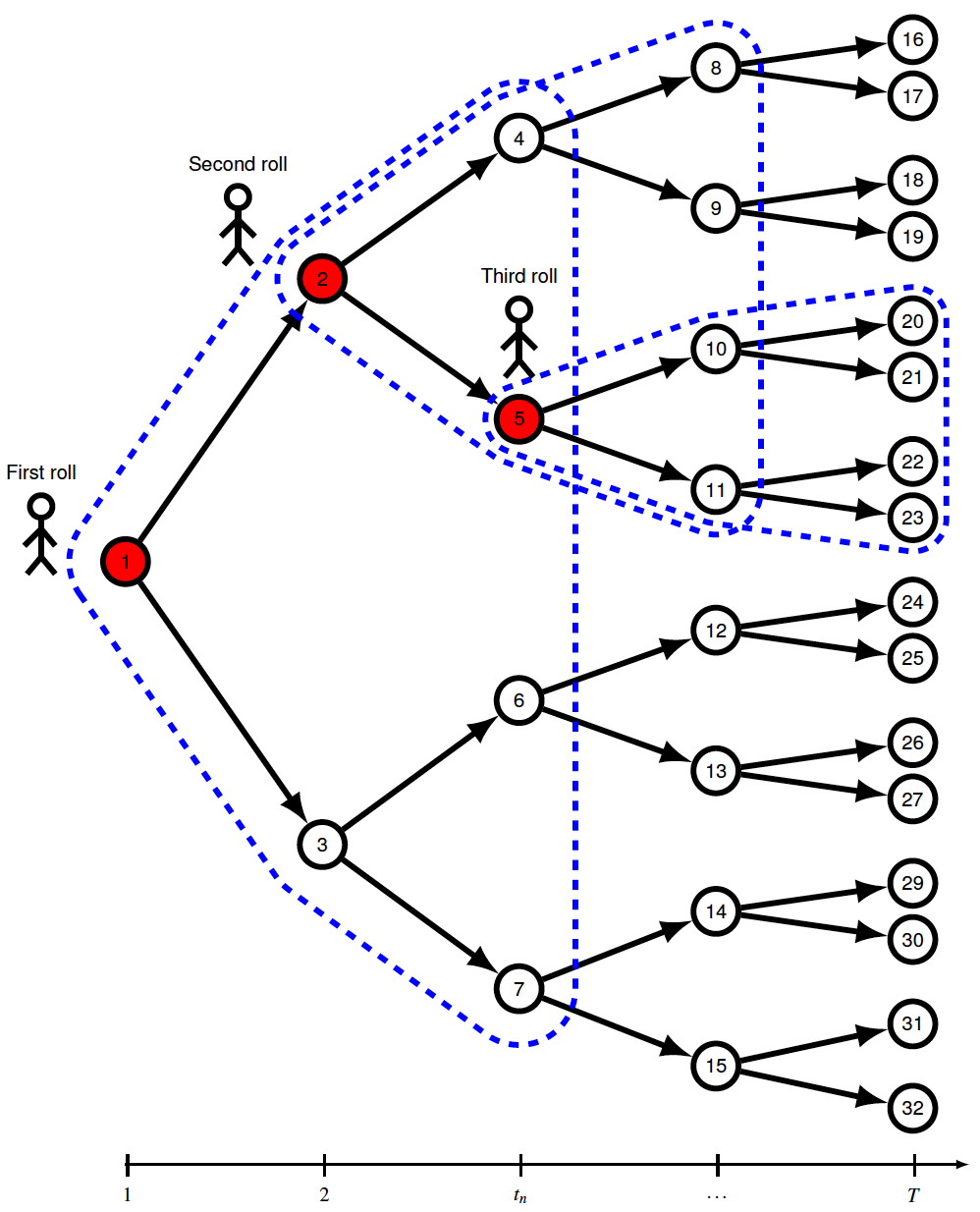}
\end{center}
\caption{Symbolic representation of the rolling horizon procedure for $T=5$ and $\tau=3$.}\label{fig:RH-procedure}
\end{figure}

%%%%%%%%%%%%%%%%%%%%%%%%%%%%%%%%%%%%%%%%%%%%%%%%%%%%%%%%%
\begin{figure}[htbp]
 \centering
  \subfigure[$\bar z$ for $|\Xi_t|= 5$ and $d_t = 1000$.]{
\includegraphics[scale=0.3]{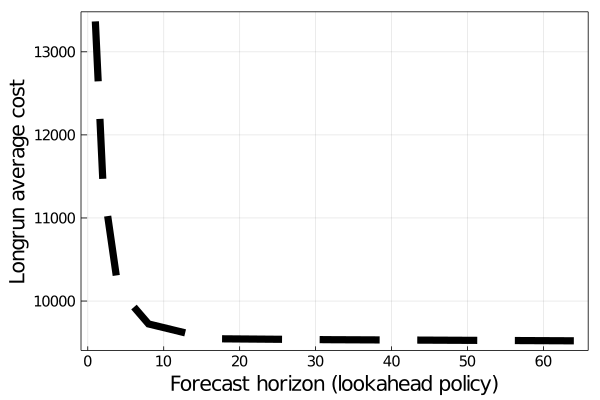}
  \label{fig:1H-5R-1000D-static_without_offline}}
  \subfigure[Training time for $|\Xi_t|= 5$ and $d_t = 1000$.]{
\includegraphics[scale=0.3]{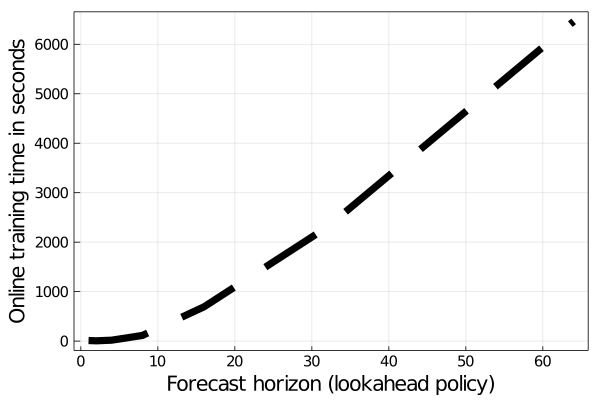}
  \label{fig:1H-5R-1000D-static_without_offline_time}}
  %%%%%%%%%%%%%%%%%%%%%%%%%%%%%%
  
  \subfigure[$\bar z$ for $|\Xi_t|= 12$ and $d_t = 1000$.]{
\includegraphics[scale=0.3]{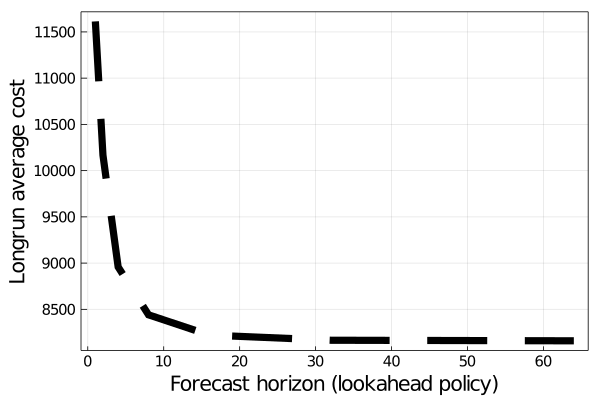}
  \label{fig:1H-12R-1000D-static_without_offline}}
  \subfigure[Training time for $|\Xi_t|= 12$ and $d_t = 1000$.]{
\includegraphics[scale=0.3]{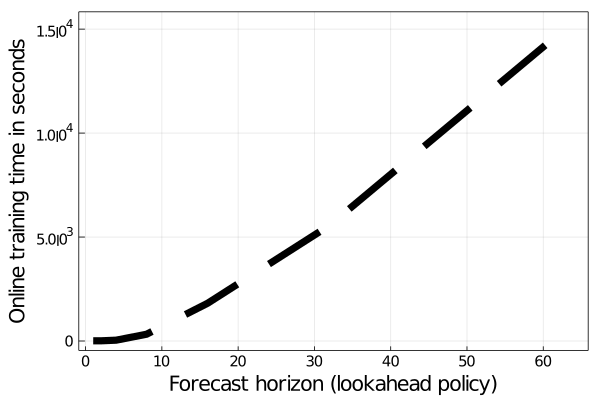}
  \label{fig:1H-12R-1000D-static_without_offline_time}}
  %%%%%%%%%%%%%%%%%%%%%%%%%%%%%%

  \subfigure[$\bar z$ for $|\Xi_t|= 5$ and $d_t = 1500$.]{
\includegraphics[scale=0.3]{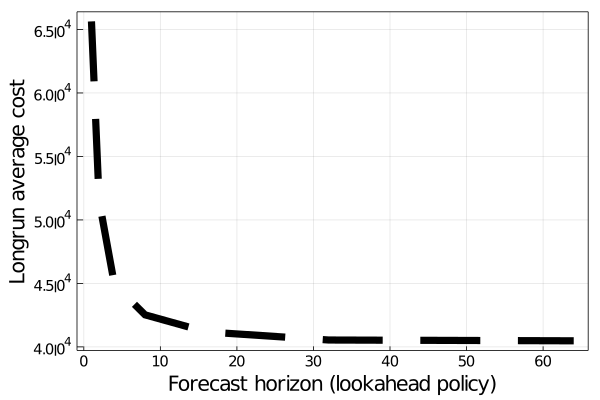}
  \label{fig:1H-5R-1500D-static_without_offline}}
  \subfigure[Training time for $|\Xi_t|= 5$ and $d_t = 1500$.]{
\includegraphics[scale=0.3]{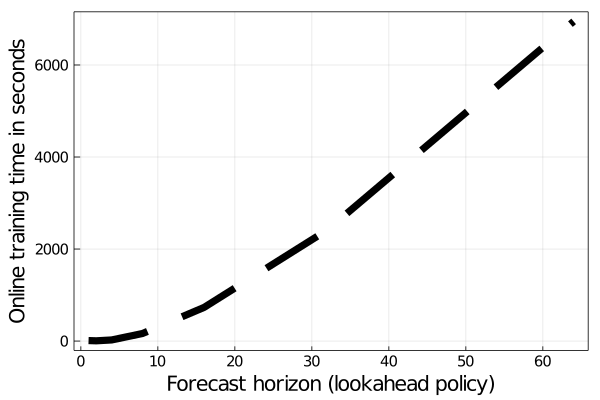}
  \label{fig:1H-5R-1500D-static_without_offline_time}}
  %%%%%%%%%%%%%%%%%%%%%%%%%%%%%%

  \subfigure[$\bar z$ for $|\Xi_t|= 12$ and $d_t = 1500$.]{
\includegraphics[scale=0.3]{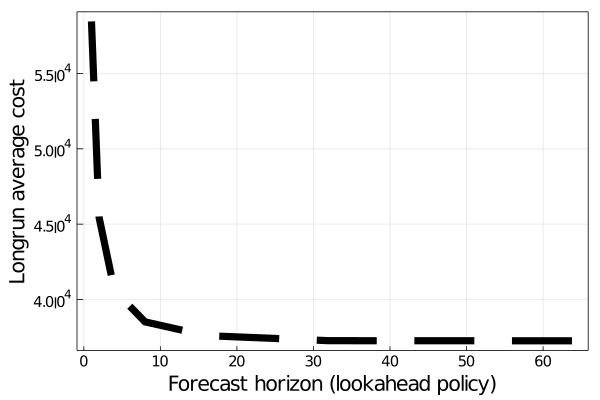}
  \label{fig:1H-12R-1500D-static_without_offline}}
  \subfigure[Training time for $|\Xi_t|= 12$ and $d_t = 1500$.]{
\includegraphics[scale=0.3]{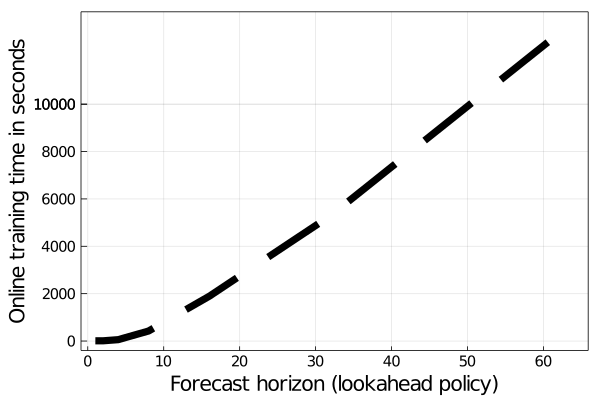}
  \label{fig:1H-12R-1500D-static_without_offline_time}}
  %%%%%%%%%%%%%%%%%%%%%%%%%%%%%%

  \hspace{0.25cm}
\caption{The long-run average cost $\bar z$ as a function of the discount factor $\gamma$ in all of the different test instances where $|H| = 1$.}
\label{fig:longrun_avg_cost_1H_static_without_offline}
\end{figure}

%%%%%%%%%%%%%%%%%%%%%%%%%%%%%%%%%%%%%%%%%%%%%%%%%%%%%%%%%
\begin{figure}[htbp]
 \centering
  \subfigure[$\bar z$ for $|\Xi_t|= 5$ and $d_t = 1750$.]{
\includegraphics[scale=0.3]{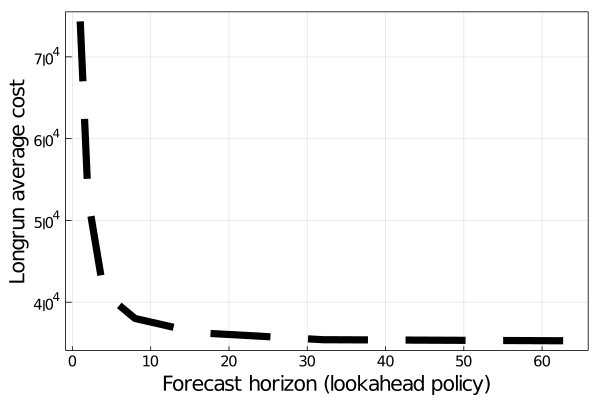}
  \label{fig:3H-5R-1750D-static_without_offline}}
  \subfigure[Training time for $|\Xi_t|= 5$ and $d_t = 1750$.]{
\includegraphics[scale=0.3]{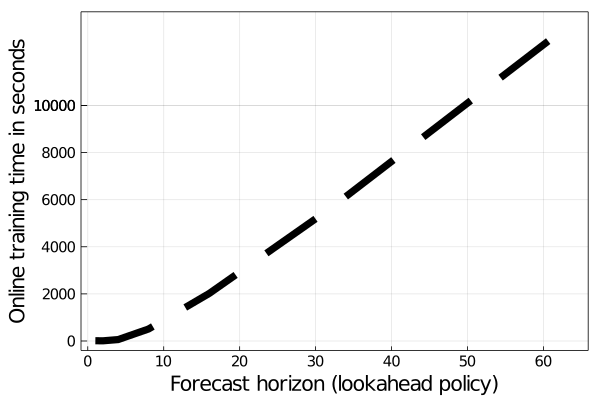}
  \label{fig:3H-5R-1750D-static_without_offline_time}}
  %%%%%%%%%%%%%%%%%%%%%%%%%%%%%%
  
  \subfigure[$\bar z$ for $|\Xi_t|= 12$ and $d_t = 1750$.]{
\includegraphics[scale=0.3]{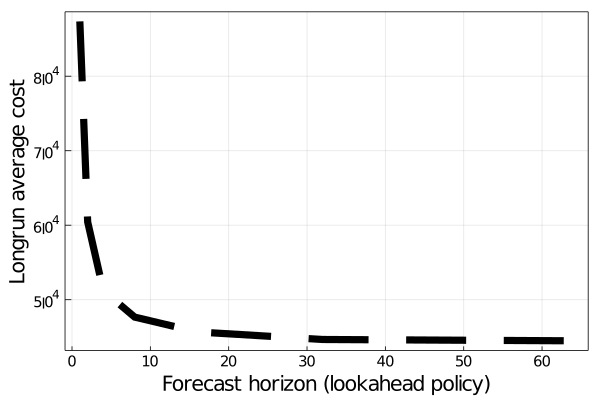}
  \label{fig:3H-12R-1750D-static_without_offline}}
  \subfigure[Training time for $|\Xi_t|= 12$ and $d_t = 1750$.]{
\includegraphics[scale=0.3]{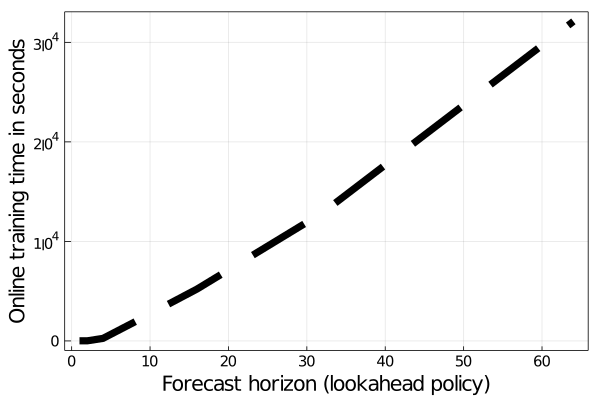}
  \label{fig:3H-12R-1750D-static_without_offline_time}}
  %%%%%%%%%%%%%%%%%%%%%%%%%%%%%%
 
  \subfigure[$\bar z$ for $|\Xi_t|= 5$ and $d_t = 2250$.]{
\includegraphics[scale=0.3]{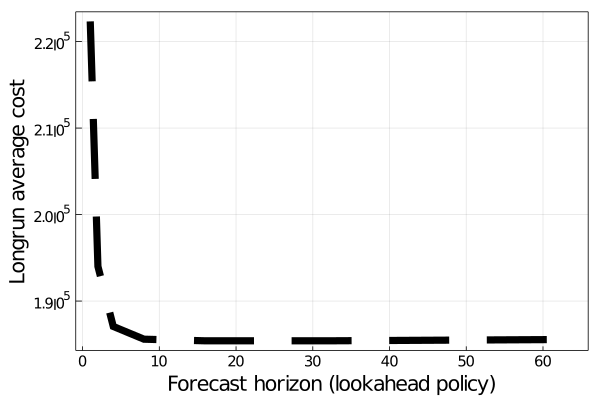}
  \label{fig:3H-5R-2250D-static_without_offline}}
  \subfigure[Training time for $|\Xi_t|= 5$ and $d_t = 2250$.]{
\includegraphics[scale=0.3]{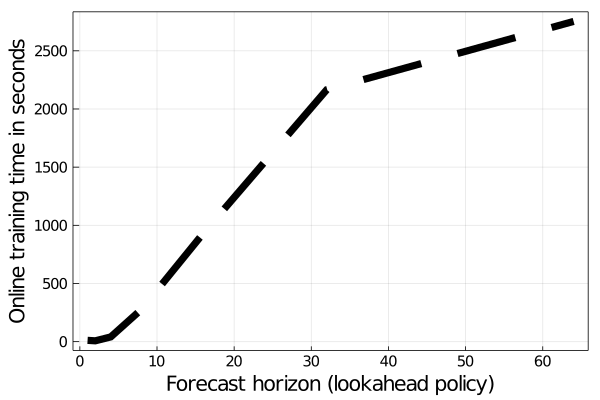}
  \label{fig:3H-5R-2250D-static_without_offline_time}}
  %%%%%%%%%%%%%%%%%%%%%%%%%%%%%%

  \subfigure[$\bar z$ for $|\Xi_t|= 12$ and $d_t = 2250$.]{
\includegraphics[scale=0.3]{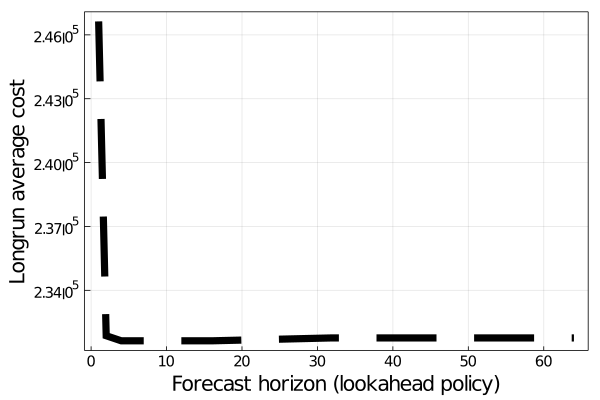}
  \label{fig:3H-12R-2250D-static_without_offline}}
  \subfigure[Training time for $|\Xi_t|= 12$ and $d_t = 2250$.]{
\includegraphics[scale=0.3]{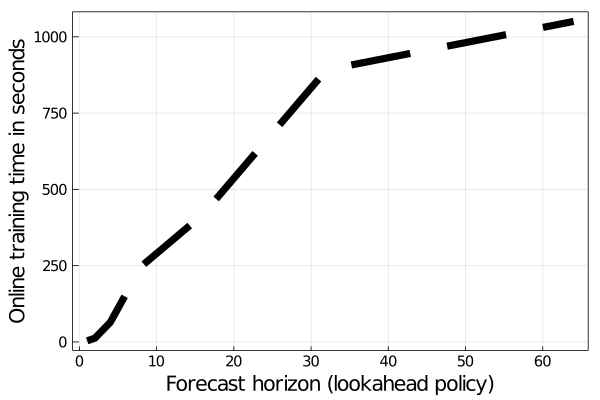}
  \label{fig:3H-12R-2250D-static_without_offline_time}}
  %%%%%%%%%%%%%%%%%%%%%%%%%%%%%%
  
  \hspace{0.25cm}
\caption{The long-run average cost $\bar z$ as a function of the discount factor $\gamma$ in all of the different test instances where $|H| = 3$.}
\label{fig:longrun_avg_cost_3H_static_without_offline}
\end{figure}

%%%%%%%%%%%%%%%%%%%%%%%%%%%%%%%%%%%%%%%%%%%%%%%%%%%%%%%%%
\begin{figure}[htbp]
 \centering
\subfigure[$\bar z$ for $|\Xi_t|= 5$ and $d_t = 2000$.]{
\includegraphics[scale=0.3]{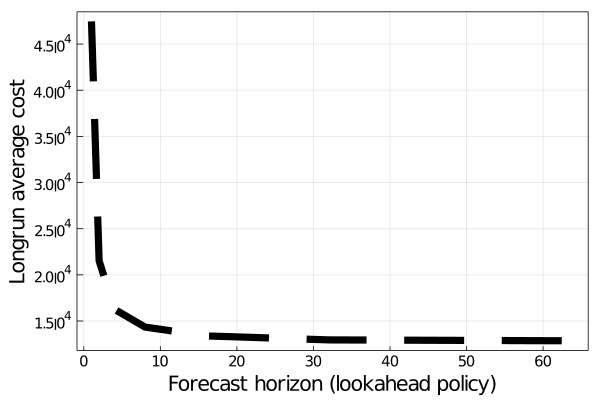}
  \label{fig:6H-5R-2000D-static_without_offline}}
  \subfigure[Training time for $|\Xi_t|= 5$ and $d_t = 2000$.]{
\includegraphics[scale=0.3]{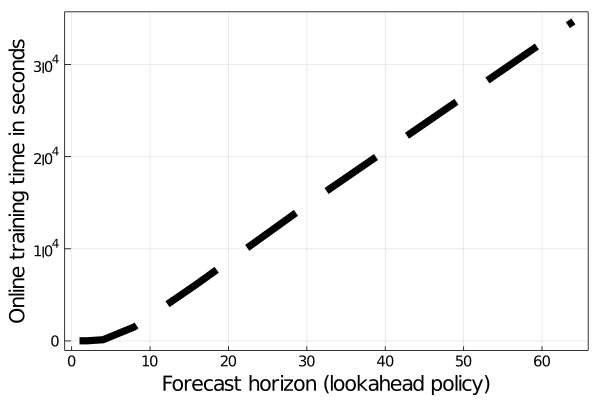}
  \label{fig:6H-5R-2000D-static_without_offline_time}}
  %%%%%%%%%%%%%%%%%%%%%%%%%%%%%%

  \subfigure[$\bar z$ for $|\Xi_t|= 12$ and $d_t = 2000$.]{
\includegraphics[scale=0.3]{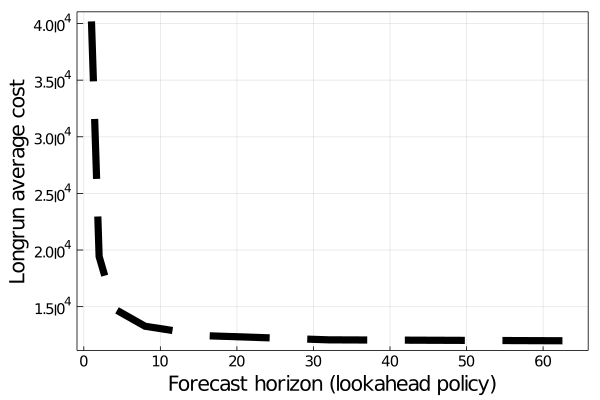}
  \label{fig:6H-12R-2000D-static_without_offline}}
  \subfigure[Training time for $|\Xi_t|= 12$ and $d_t = 2000$.]{
\includegraphics[scale=0.3]{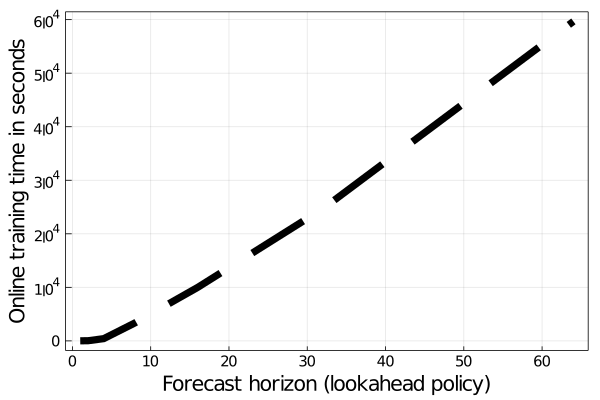}
  \label{fig:6H-12R-2000D-static_without_offline_time}}
  %%%%%%%%%%%%%%%%%%%%%%%%%%%%%%

  \subfigure[$\bar z$ for $|\Xi_t|= 5$ and $d_t = 2250$.]{
\includegraphics[scale=0.3]{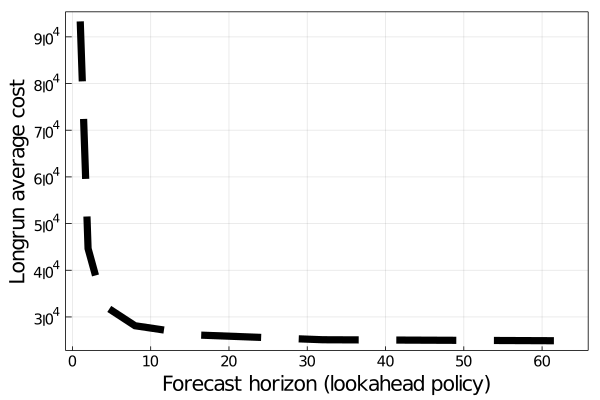}
  \label{fig:6H-5R-2250D-static_without_offline}}
  \subfigure[Training time for $|\Xi_t|= 5$ and $d_t = 2250$.]{
\includegraphics[scale=0.3]{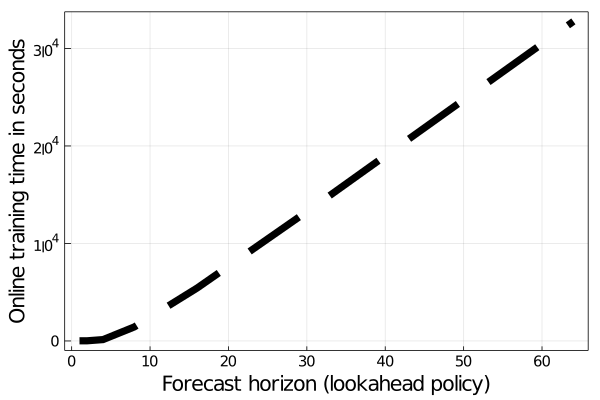}
  \label{fig:6H-5R-2250D-static_without_offline_time}}
  %%%%%%%%%%%%%%%%%%%%%%%%%%%%%%

  \subfigure[$\bar z$ for $|\Xi_t|= 12$ and $d_t = 2250$.]{
\includegraphics[scale=0.3]{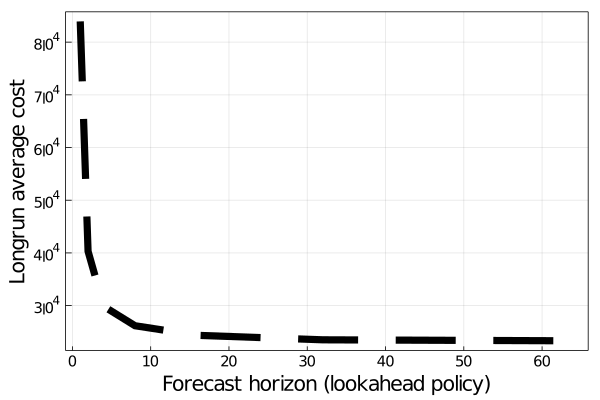}
  \label{fig:6H-12R-2250D-static_without_offline}}
  \subfigure[Training time for $|\Xi_t|= 12$ and $d_t = 2250$.]{
\includegraphics[scale=0.3]{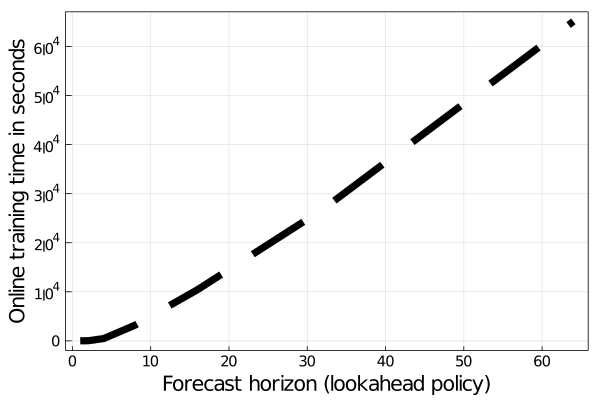}
  \label{fig:6H-12R-2250D-static_without_offline_time}}
  %%%%%%%%%%%%%%%%%%%%%%%%%%%%%%

  \hspace{0.25cm}
\caption{The long-run average cost $\bar z$ as a function of the discount factor $\gamma$ in all of the different test instances where $|H| = 6$.}
\label{fig:longrun_avg_cost_6H_static_without_offline}
\end{figure}

%%%%%%%%%%%%%%%%%%%%%%%%%%%%%%%%%%%%%%%%%%%%%%%%%%%%%%%%%
\begin{figure}[htbp]
 \centering
  \subfigure[$\bar z$ for $|\Xi_t|= 5$ and $d_t = 1000$.]{
\includegraphics[scale=0.3]{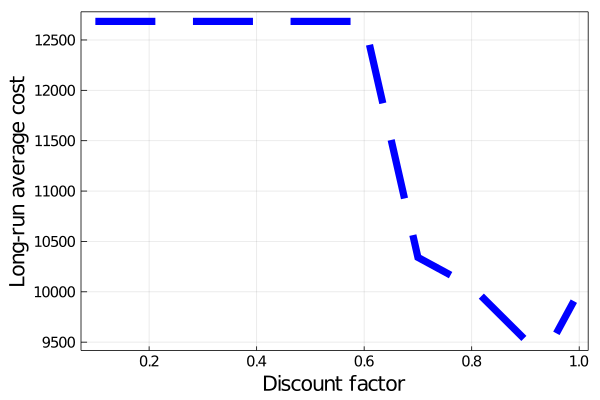}
  \label{fig:1H-5R-1000D-discounted}}
  \subfigure[Training time for $|\Xi_t|= 5$ and $d_t = 1000$.]{
\includegraphics[scale=0.3]{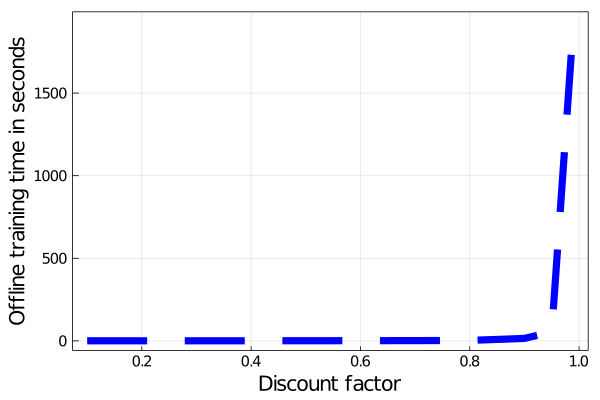}
  \label{fig:1H-5R-1000D-discounted_time}}
  %%%%%%%%%%%%%%%%%%%%%%%%%%%%%%
  
  \subfigure[$\bar z$ for $|\Xi_t|= 12$ and $d_t = 1000$.]{
\includegraphics[scale=0.3]{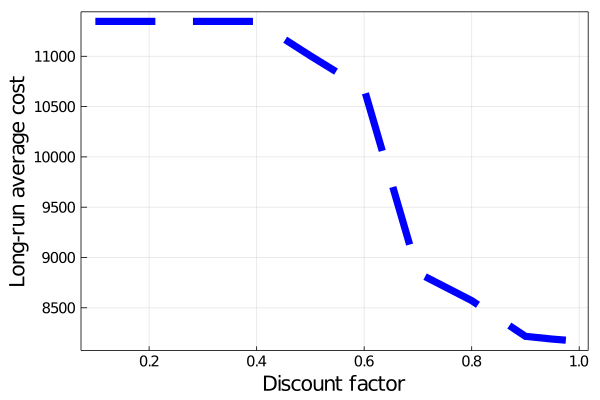}
  \label{fig:1H-12R-1000D-discounted}}
  \subfigure[Training time for $|\Xi_t|= 12$ and $d_t = 1000$.]{
\includegraphics[scale=0.3]{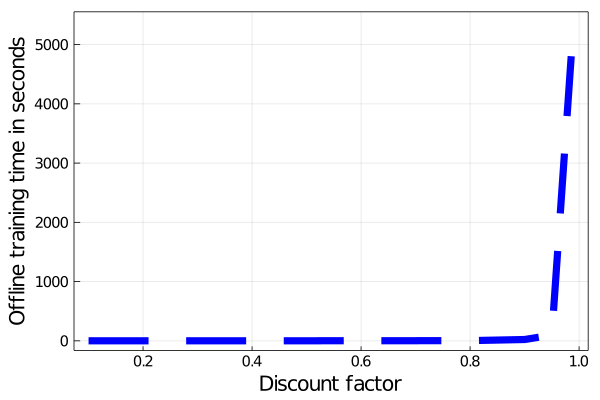}
  \label{fig:1H-12R-1000D-discounted_time}}
  %%%%%%%%%%%%%%%%%%%%%%%%%%%%%%

  \subfigure[$\bar z$ for $|\Xi_t|= 5$ and $d_t = 1500$.]{
\includegraphics[scale=0.3]{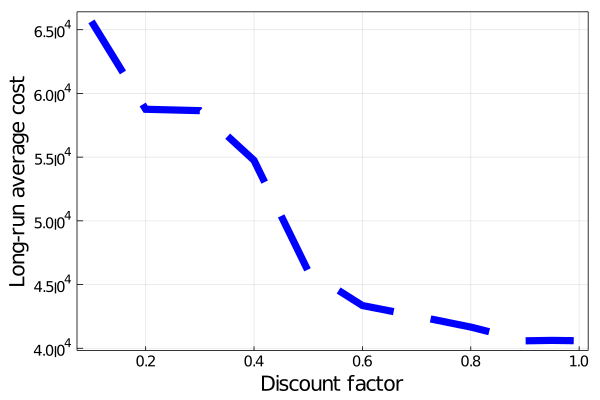}
  \label{fig:1H-5R-1500D-discounted}}
  \subfigure[Training time for $|\Xi_t|= 5$ and $d_t = 1500$.]{
\includegraphics[scale=0.3]{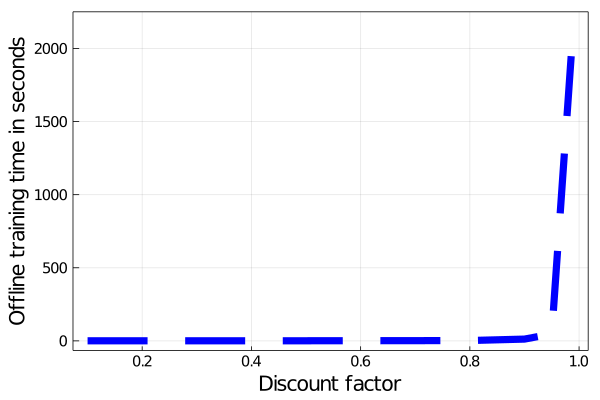}
  \label{fig:1H-5R-1500D-discounted_time}}
  %%%%%%%%%%%%%%%%%%%%%%%%%%%%%%

  \subfigure[$\bar z$ for $|\Xi_t|= 12$ and $d_t = 1500$.]{
\includegraphics[scale=0.3]{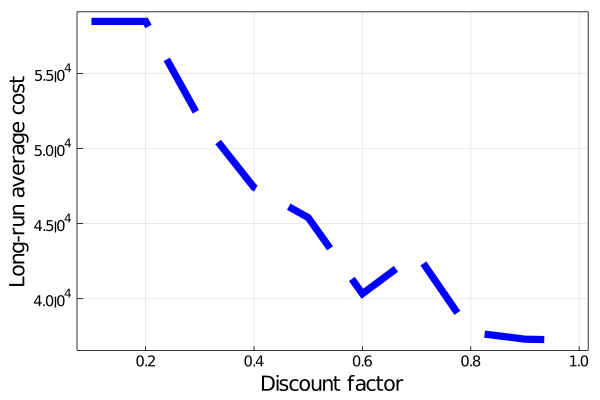}
  \label{fig:1H-12R-1500D-discounted}}
  \subfigure[Training time for $|\Xi_t|= 12$ and $d_t = 1500$.]{
\includegraphics[scale=0.3]{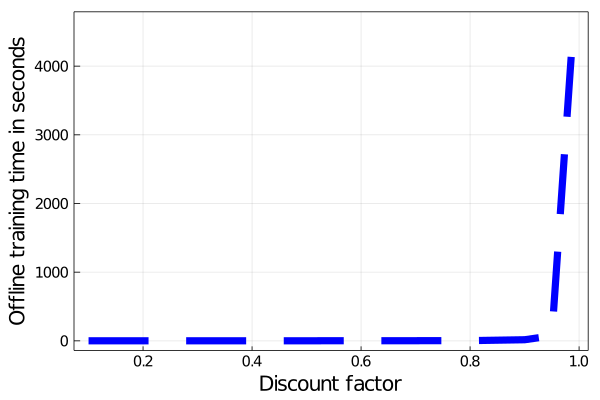}
  \label{fig:1H-12R-1500D-discounted_time}}
  %%%%%%%%%%%%%%%%%%%%%%%%%%%%%%

  \hspace{0.25cm}
\caption{The long-run average cost $\bar z$ as a function of the discount factor $\gamma$ in all of the different test instances where $|H| = 1$.}
\label{fig:longrun_avg_costH_discount}
\end{figure}

%%%%%%%%%%%%%%%%%%%%%%%%%%%%%%%%%%%%%%%%%%%%%%%%%%%%%%%%%
\begin{figure}[htbp]
 \centering
  \subfigure[$\bar z$ for $|\Xi_t|= 5$ and $d_t = 1750$.]{
\includegraphics[scale=0.3]{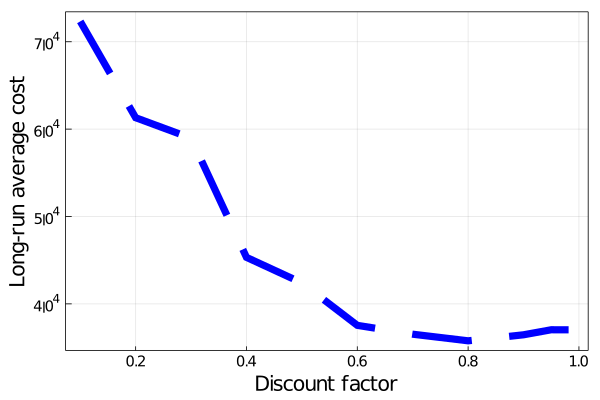}
  \label{fig:3H-5R-1750D-discounted}}
  \subfigure[Training time for $|\Xi_t|= 5$ and $d_t = 1750$.]{
\includegraphics[scale=0.3]{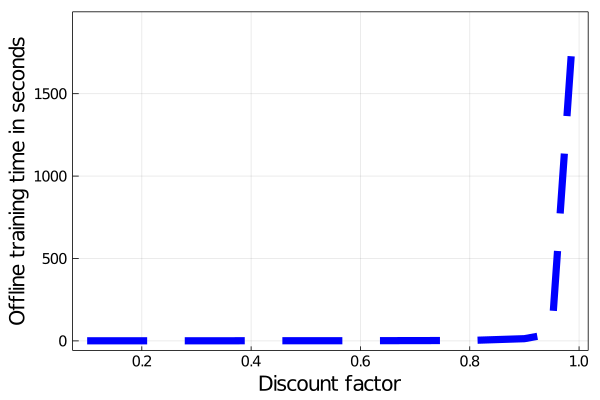}
  \label{fig:3H-5R-1750D-discounted_time}}
  %%%%%%%%%%%%%%%%%%%%%%%%%%%%%%
  
  \subfigure[$\bar z$ for $|\Xi_t|= 12$ and $d_t = 1750$.]{
\includegraphics[scale=0.3]{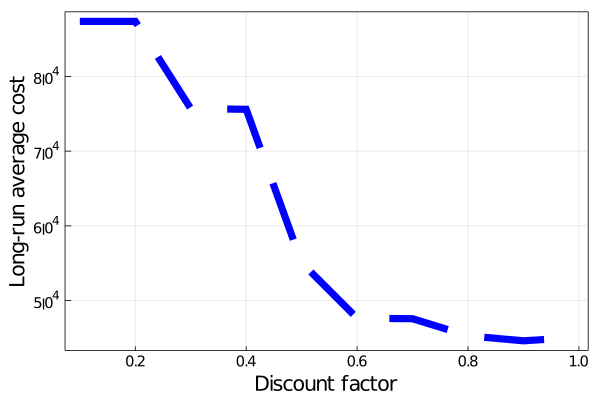}
  \label{fig:3H-12R-1750D-discounted}}
  \subfigure[Training time for $|\Xi_t|= 12$ and $d_t = 1750$.]{
\includegraphics[scale=0.3]{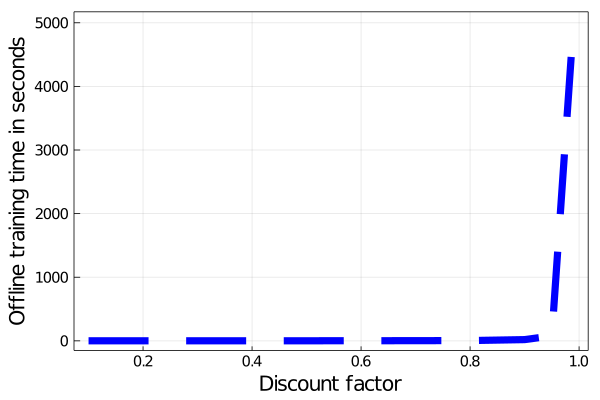}
  \label{fig:3H-12R-1750D-discounted_time}}
  %%%%%%%%%%%%%%%%%%%%%%%%%%%%%%
 
  \subfigure[$\bar z$ for $|\Xi_t|= 5$ and $d_t = 2250$.]{
\includegraphics[scale=0.3]{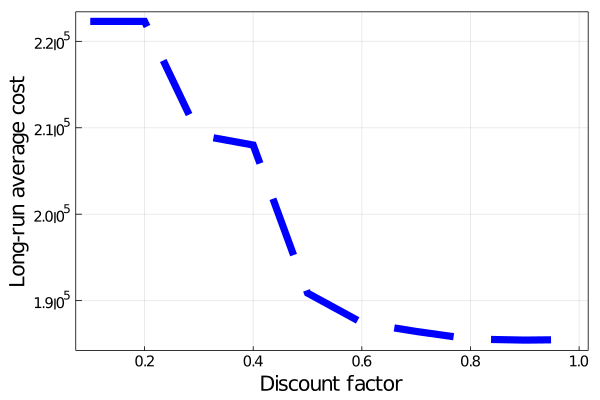}
  \label{fig:3H-5R-2250D-discounted}}
  \subfigure[Training time for $|\Xi_t|= 5$ and $d_t = 2250$.]{
\includegraphics[scale=0.3]{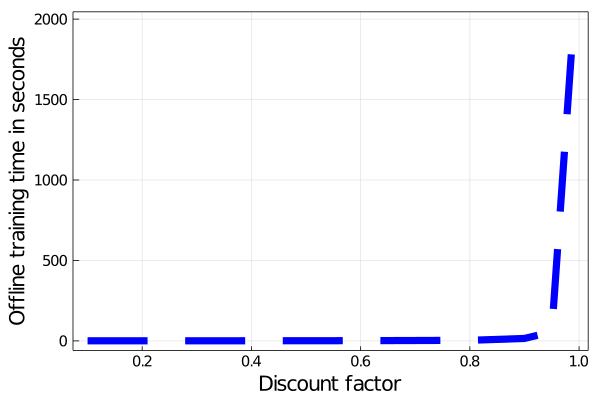}
  \label{fig:3H-5R-2250D-discounted_time}}
  %%%%%%%%%%%%%%%%%%%%%%%%%%%%%%

  \subfigure[$\bar z$ for $|\Xi_t|= 12$ and $d_t = 2250$.]{
\includegraphics[scale=0.3]{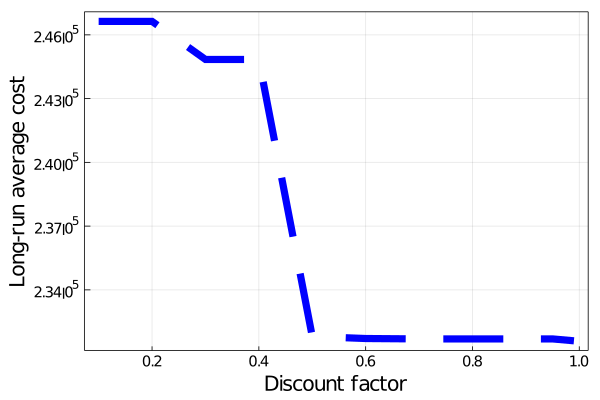}
  \label{fig:3H-12R-2250D-discounted}}
  \subfigure[Training time for $|\Xi_t|= 12$ and $d_t = 2250$.]{
\includegraphics[scale=0.3]{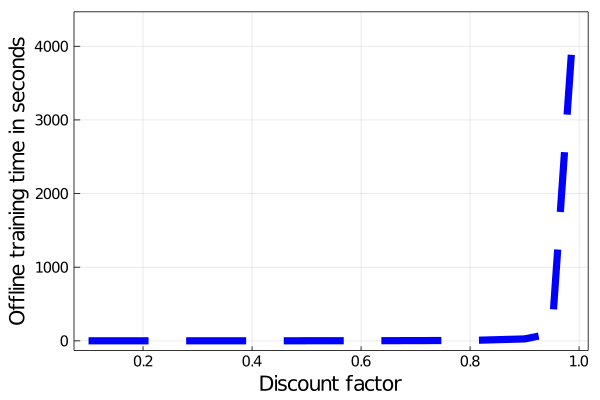}
  \label{fig:3H-12R-2250D-discounted_time}}
  %%%%%%%%%%%%%%%%%%%%%%%%%%%%%%
  
  \hspace{0.25cm}
\caption{The long-run average cost $\bar z$ as a function of the discount factor $\gamma$ in all of the different test instances where $|H| = 3$.}
\label{fig:longrun_avg_cost_3H_discount}
\end{figure}

%%%%%%%%%%%%%%%%%%%%%%%%%%%%%%%%%%%%%%%%%%%%%%%%%%%%%%%%%
\begin{figure}[htbp]
 \centering
\subfigure[$\bar z$ for $|\Xi_t|= 5$ and $d_t = 2000$.]{
\includegraphics[scale=0.3]{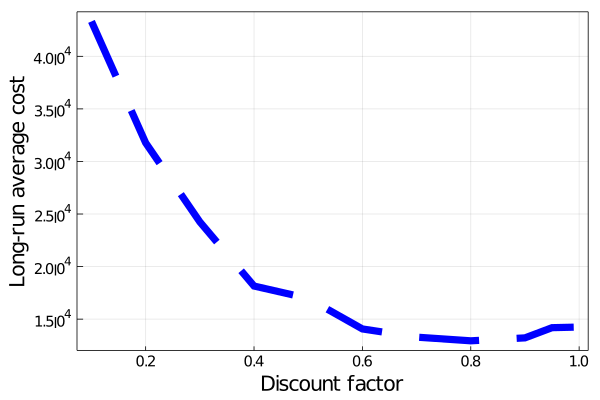}
  \label{fig:6H-5R-2000D-discounted}}
  \subfigure[Training time for $|\Xi_t|= 5$ and $d_t = 2000$.]{
\includegraphics[scale=0.3]{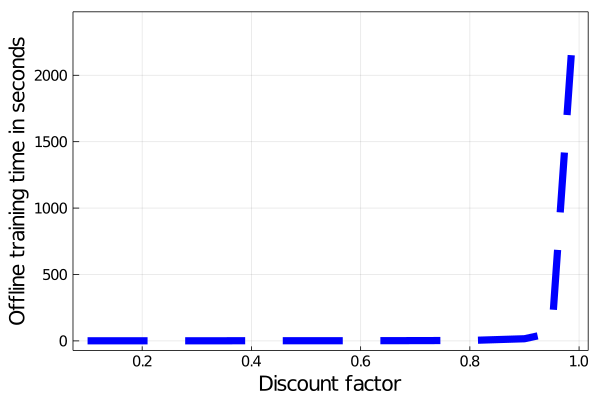}
  \label{fig:6H-5R-2000D-discounted_time}}
  %%%%%%%%%%%%%%%%%%%%%%%%%%%%%%

  \subfigure[$\bar z$ for $|\Xi_t|= 12$ and $d_t = 2000$.]{
\includegraphics[scale=0.3]{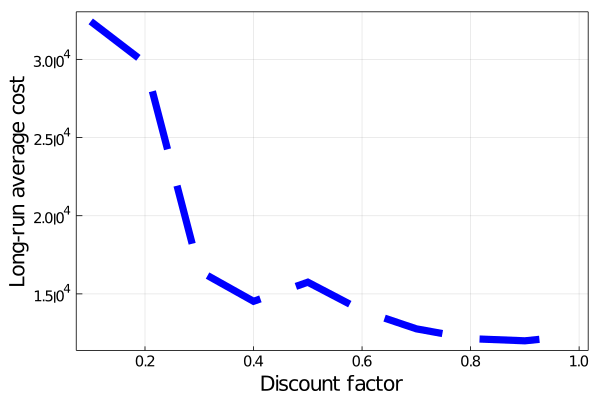}
  \label{fig:6H-12R-2000D-discounted}}
  \subfigure[Training time for $|\Xi_t|= 12$ and $d_t = 2000$.]{
\includegraphics[scale=0.3]{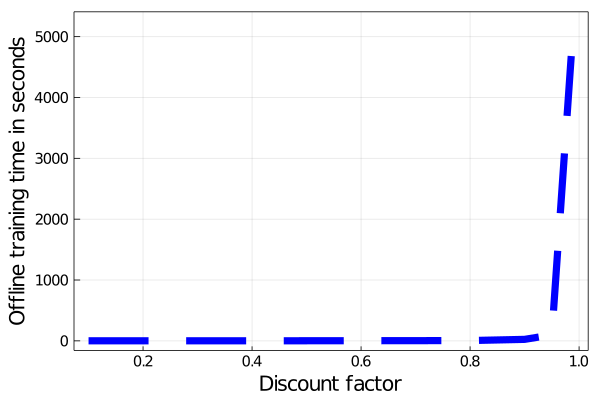}
  \label{fig:6H-12R-2000D-discounted_time}}
  %%%%%%%%%%%%%%%%%%%%%%%%%%%%%%

  \subfigure[$\bar z$ for $|\Xi_t|= 5$ and $d_t = 2250$.]{
\includegraphics[scale=0.3]{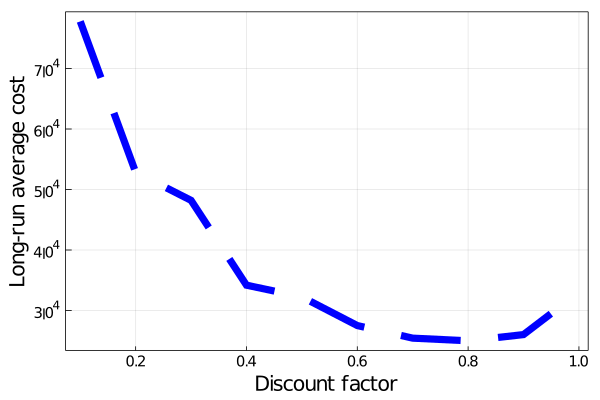}
  \label{fig:6H-5R-2250D-discounted}}
  \subfigure[Training time for $|\Xi_t|= 5$ and $d_t = 2250$.]{
\includegraphics[scale=0.3]{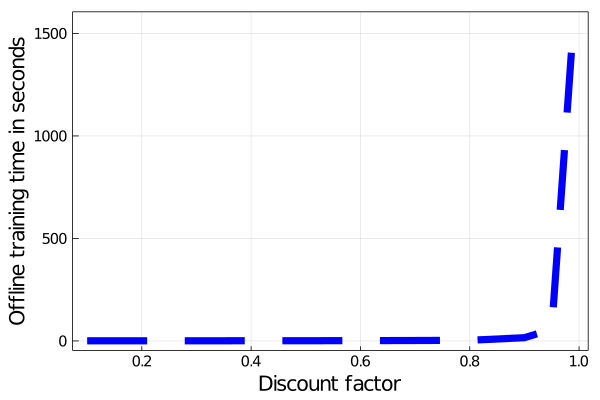}
  \label{fig:6H-5R-2250D-discounted_time}}
  %%%%%%%%%%%%%%%%%%%%%%%%%%%%%%

  \subfigure[$\bar z$ for $|\Xi_t|= 12$ and $d_t = 2250$.]{
\includegraphics[scale=0.3]{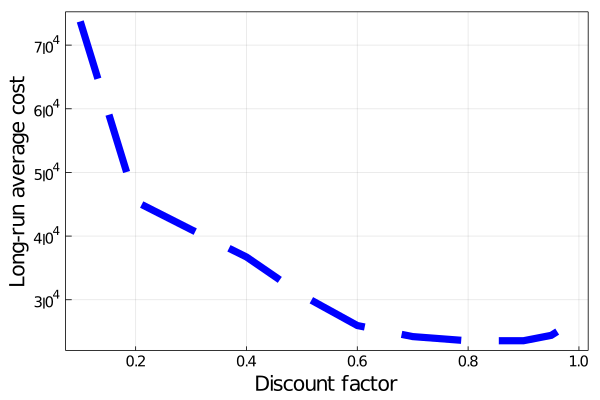}
  \label{fig:6H-12R-2250D-discounted}}
  \subfigure[Training time for $|\Xi_t|= 12$ and $d_t = 2250$.]{
\includegraphics[scale=0.3]{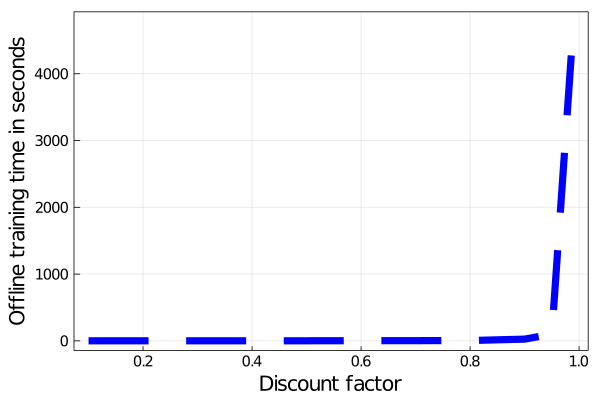}
  \label{fig:6H-12R-2250D-discounted_time}}
  %%%%%%%%%%%%%%%%%%%%%%%%%%%%%%

  \hspace{0.25cm}
\caption{The long-run average cost $\bar z$ as a function of the discount factor $\gamma$ in all of the different test instances where $|H| = 6$.}
\label{fig:longrun_avg_cost_6H_discount}
\end{figure}

\end{document}